\documentclass[11pt, reqno, english]{amsart}
\usepackage{accents}
\usepackage{style}
\usepackage[normalem]{ulem}
\usepackage{wasysym}
\usepackage{charter}
\usepackage{euler}
\usepackage[T1]{fontenc}
\usepackage{thmtools}
\usepackage{thm-restate}
\newcommand{\og}{\overline{g}}
\newcommand{\oh}{\overline{h}}
\newcommand{\of}{\overline{f}}

\def\U{\mathcal U}
\def\N{\mathcal N}
\def\fr{\mathrm{FillRad}}

\begin{document}

\title{
Hausdorff vs Gromov--Hausdorff distances
}

\author{Henry Adams}
\email{henry.adams@ufl.edu}
\author{Florian Frick}
\email{frick@cmu.edu}
\author{Sushovan Majhi}
\email{s.majhi@gwu.edu}
\author{Nicholas McBride}
\email{nmcbride@ucsc.edu}

\begin{abstract}
Let $M$ be a closed Riemannian manifold and let $X\subseteq M$.
If the sample $X$ is sufficiently dense relative to the curvature of $M$, then the Gromov--Hausdorff distance between $X$ and $M$ is bounded from below by half their Hausdorff distance, namely $d_\gh(X,M) \ge \frac{1}{2} d_\h(X,M)$.
The constant $\frac{1}{2}$ can be improved depending on the dimension and curvature of the manifold $M$, and obtains the optimal value $1$ in the case of the unit circle, meaning that if $X\subseteq S^1$ satisfies $d_\gh(X,S^1)<\tfrac{\pi}{6}$, then $d_\gh(X,S^1)=d_\h(X,S^1)$.
We also provide versions lower bounding the Gromov--Hausdorff distance $d_\gh(X,Y)$ between two subsets $X,Y\subseteq M$.
Our proofs convert discontinuous functions between metric spaces into simplicial maps between \v{C}ech or Vietoris--Rips complexes.
We then produce topological obstructions to the existence of certain maps using the nerve lemma and the fundamental class of the manifold, thus lower bounding the Gromov--Hausdorff distance.
\end{abstract}


\maketitle


\section{Introduction}
The Gromov--Hausdorff distance, denoted $d_\gh(X,Y)$, gives rise to a natural dissimilarity measure between two abstract metric spaces $X$ and $Y$ not necessarily sitting in a common ambient space~\cite{edwards1975structure,gromov1981groups, gromov1981structures}.
Despite its increasing popularity and usefulness in the theoretical and applied communities, much is yet to be discovered about the Gromov--Hausdorff distance. 

One such question, for $X$ a subset of a metric space $M$, is how $d_\gh(X,M)$ relates to the Hausdorff distance $d_\h(X,M)$. 
This paper furthers the understanding of the relationship between the Hausdorff distance and Gromov--Hausdorff distance when $M$ is a closed Riemannian manifold.
Furthermore, we provide new lower bounds on $d_\gh(X,Y)$ when $X$ and $Y$ are dense-enough subsets of a manifold~$M$.

\begin{figure}[htb]
\centering
\includegraphics[width=4in]{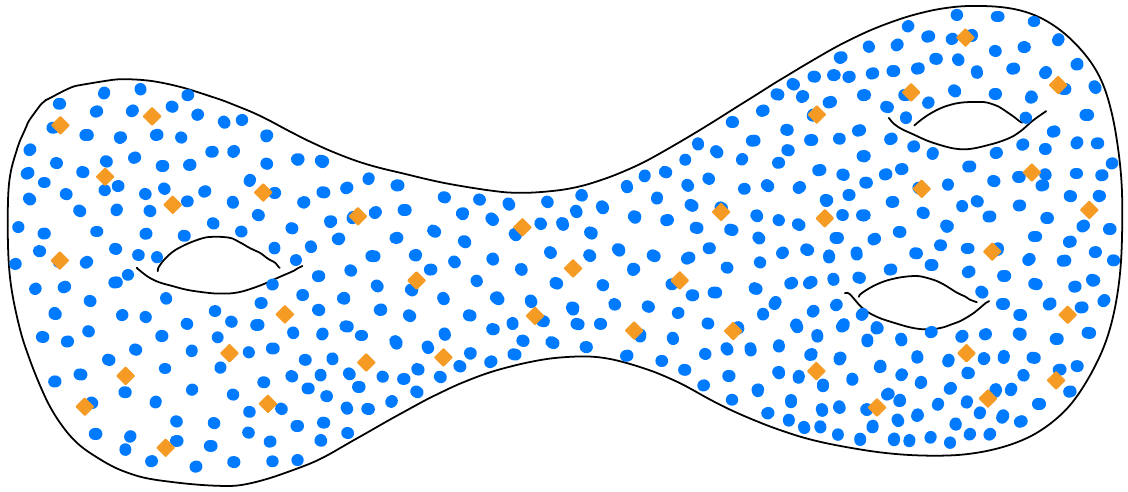}
\caption{A manifold $M$, a subset $X$ (orange diamonds), and a subset $Y$ (blue circles).}
\label{fig:two-subsets}
\end{figure}

By definition, the Gromov--Hausdorff distance between $X$ and $M$ is bounded from above by the Hausdorff distance between them: $d_\gh(X,M) \le
d_\h(X,M)$.
We show that if the sample $X$ is sufficiently dense relative to the curvature of a manifold~$M$, then the Gromov--Hausdorff distance is also bounded from below by half the Hausdorff distance: $d_\gh(X,M) \ge
\frac{1}{2} \cdot d_\h(X,M)$. 
Thus, the asymptotics of the Gromov--Hausdorff distance, whose definition involves isometric embeddings into arbitrary metric spaces, are controlled by the Hausdorff distance in a fixed space.
Furthermore, we prove that if $X$ and $Y$ are two subsets of $M$ and dense enough relative to the convexity radius of $M$, then $d_\gh(X,Y) \ge \frac{1}{2}d_\h(X,M)-d_\h(Y,M)$ (see Figure~\ref{fig:two-subsets}).

\vspace{3mm}
\begin{restatable}{theorem-main}{mainOne}
\label{thm:main1}
Let $M$ be a connected, closed Riemannian manifold with convexity radius $\rho(M)$.
Then for any $X,Y\subseteq M$, we have
\begin{enumerate}[(a)]
\item $d_\gh(X,M) \ge \min\left\{\tfrac{1}{2}d_\h(X,M), \tfrac{1}{6}\rho(M)\right\}$, and
\item $d_\gh(X,Y) \ge \min\left\{\tfrac{1}{2}d_\h(X,M)-d_\h(Y,M), \tfrac{1}{6}\rho(M)-\frac{2}{3}d_\h(Y,M)\right\}$.
\end{enumerate}
\end{restatable}
\vspace{3mm}

The interpretation of part (a) in Theorem~\ref{thm:main1} is as follows.
If $X\subseteq M$ is sufficiently dense (meaning $d_\gh(X,M)<\tfrac{1}{6}\rho(M)$), then $d_\gh(X,M) \ge \tfrac{1}{2}d_\h(X,M)$.

Part (b) above is trivially true unless $Y$ is more than twice as dense as $X$, i.e.\ unless $d_\h(Y,M)< \frac{1}{2}d_\h(X,M)$. 

Theorem~\ref{thm:main1}(b) implies an important corollary, which is motivated by the computational aspects of the lower bounds we provide for the Gromov--Hausdorff distance.
In Corollary~\ref{cor:two-subsets-eps}, we show that if $X, Y$ are dense enough subsets of $M$ with $d_\h(Y,M)\le\varepsilon$, then 
\[
d_\gh(X,Y) \ge \min\left\{\tfrac{1}{2}d_\h(X,Y)-\tfrac{3}{2}\varepsilon, \tfrac{1}{6}\rho(M)-\tfrac{2}{3}\varepsilon\right\}.
\]
This result establishes a computationally feasible lower bound for $d_\gh(X,Y)$ in terms of the Hausdorff distance $d_\h(X,Y)$ of the two subsets.
Moreover, if  $X$ and $Y$ are finite subsets with at most $n$ points, then $d_\h(X,Y)$ is computable in $\mathcal{O}(n^2)$-time.


We encourage the reader to think of Theorem~\ref{thm:main1}(a) as our main result, with a short but novel proof (see Section~\ref{sec:manifold-simple}) that is approachable even for newcomers to this research area.
Our proof uses \v{C}ech simplicial complexes in order to produce topological obstructions to the existence of functions between $X$ and $M$ of low distortion and codistortion, thus lower bounding the Gromov--Hausdorff distance.
The topology of \v{C}ech complexes appears in three main ways.
First, it allows us to turn possibly discontinuous functions between $X$ and $M$ into continuous (simplicial) maps between \v{C}ech complexes on $X$ and on $M$.
Second, our quantitative bound on $X\subseteq M$ being ``sufficiently dense'' depends on the convexity radius of the manifold $M$, which arises since the nerve lemma implies that the \v{C}ech complex of $M$ is homotopy equivalent to $M$ for scale parameters up to the convexity radius.
Third, the nerve lemma~\cite{Alexandroff1928,Borsuk1948} implies that the \v{C}ech complex of $X\subseteq M$ does not recover the fundamental class of $M$ until the scale is at least the Hausdorff distance $d_\h(X,M)$.
Our proof technique for Theorem~\ref{thm:main1}(a) is flexible: the remaining results in our paper show that various terms in part (a) can be improved by incorporating fancier tools into the same proof strategy.

The constants in Theorem~\ref{thm:main1} are not optimal, and we improve them by incorporating Vietoris--Rips complexes into the proofs.
Before showing how to improve the constants for a general manifold, we first show how to improve them in the case when $M=S^1$ is the circle.

\vspace{3mm}
\begin{restatable}{theorem-main}{mainTwo}
\label{thm:main2}    
For any subsets $X,Y\subseteq S^1$ of the circle of circumference $2\pi$, we have
\begin{enumerate}[(a)]
\item $d_\gh(X,S^1)\geq\min\left\{d_\h(X,S^1),\tfrac{\pi}{6}\right\}$, and
\item $d_\gh(X,Y)\geq\min\left\{d_\h(X,S^1)-d_\h(Y,S^1),\tfrac{\pi}{6}-\tfrac{1}{2}d_\h(Y,S^1)\right\}$.
\end{enumerate}
\end{restatable}
\vspace{3mm}

Part (a) in Theorem~\ref{thm:main2} is obtained from part (b) simply by letting $Y=S^1$.
Since the Gromov--Hausdorff distance is never larger than the Hausdorff distance, if $X\subseteq S^1$ is sufficiently dense (meaning $d_\gh(X,S^1)<\tfrac{\pi}{6}$), then we have an equality $d_\gh(X,S^1)=d_\h(X,S^1)$ between the Gromov--Hausdorff distance and the Hausdorff distance.

When applying part (b) above, one should always choose $X$ and $Y$ with $d_\h(Y,S^1)\le d_\h(X,S^1)$. 
If both $X,Y\subseteq S^1$ are sufficiently dense, i.e.\ $d_\h(Y,S^1) \le d_\h(X,S^1) \le \frac{\pi}{6}$, then the first term achieves the minimum, implying $d_\gh(X,Y) \ge |d_\h(X,S^1)-d_\h(Y,S^1)|$.

We next show how to improve the bounds in Theorem~\ref{thm:main1} for a general manifold $M$, not only for a circle.
The first term to improve is the constant in front of the convexity radius $\rho(M)$, which we improve using Gromov's filling radius~\cite{gromov1983filling,lim2020vietoris}.

\vspace{3mm}
\begin{restatable}{theorem-main}{mainThree}
\label{thm:main3}
Let $M$ be a connected, closed Riemannian manifold with filling radius $\fr(M)$.
Then for any $X,Y\subseteq M$, we have
\begin{enumerate}[(a)]
\item $d_\gh(X,M) \ge \min\left\{\tfrac{1}{2}d_\h(X,M), \tfrac{1}{2}\rho(M), \tfrac{1}{3}\fr(M)\right\}$, and
\item \ 
\small
\[\hspace{-6mm}d_\gh(X,Y) \ge \min\left\{\tfrac{1}{2}d_\h(X,M)-d_\h(Y,M), \tfrac{1}{2}\rho(M)-d_\h(Y,M), \tfrac{1}{3}\fr(M)-\tfrac{2}{3}d_\h(Y,M)\right\}.\]
\normalsize
\end{enumerate}
\end{restatable}
\vspace{3mm}

Part (a) in Theorem~\ref{thm:main3} is obtained from part (b) simply by letting $Y=M$.

Next, in Theorem~\ref{thm:main4} we show how to improve the factor of $\frac{1}{2}$ in front of the $d_\h(X,M)$ term in Theorem~\ref{thm:main1} to a constant $\alpha(n,\kappa)$ which depends only on the dimension $n$ of the manifold and on the upper bound $\kappa$ of the sectional curvatures on $M$.
These constants are equal to $\alpha(n,\kappa)=\sqrt{\tfrac{n+1}{2n}}$ if $\kappa\leq0$.
The proof of Theorem~\ref{thm:main4} enables better bounds because it incorporates not only \v{C}ech complexes but also Vietoris--Rips complexes, allowing us to incorporate Jung's theorem on manifolds~\cite{danzer1963helly,Dekster1985AnEO,Dekster1995TheJT,Dekster1997}.

\vspace{3mm}
\begin{restatable}{theorem-main}{mainFour}
\label{thm:main4}
Let $M$ be a connected, closed Riemannian $n$-manifold with convexity radius $\rho(M)$ and an upper bound $\kappa$ on all sectional curvatures at all points of $M$. 
For any subsets $X,Y\subseteq M$,
\[d_\gh(X,Y) \geq \min\left\{\alpha(n,\kappa)d_\h(X,M)-d_\h(Y,M),\frac{\alpha(n,\kappa)\tau-2d_\h(Y,M)}{2\alpha(n,\kappa)+2}\right\},\]
where $\tau=\rho(M)$ if $\kappa\le 0$ and $\tau=\min\left\{\rho(M),\tfrac{\pi}{2\sqrt{\kappa+1}}\right\}$ if $\kappa>0$.
\end{restatable}
\vspace{3mm}

These lower bounds on the Gromov--Hausdorff distance are quite different than the general setting, since in Theorem~\ref{thm:main5} we show that if $Z$ is a general compact metric space (not necessarily a manifold) and $X$ is a subset thereof (not necessarily sufficiently dense), then the ratio between $d_\gh(X,Z)$ and $d_\h(X,Z)$ can be arbitrarily far away from one.

\vspace{3mm}
\begin{restatable}{theorem-main}{mainFive}
\label{thm:main5}    
For any $\varepsilon>0$, there exists a compact metric space $Z$ and a subset $X\subseteq Z$ with \[d_\gh(X,Z) < \varepsilon \cdot d_\h(X,Z).\]
\end{restatable}
\vspace{3mm}

In general the Gromov--Hausdorff distance is NP-hard to compute or approximate~\cite{agarwal_computing_2015,Schmiedl2017}.
Though our paper is not algorithmic, it points to situations where one may be able to approximate Gromov--Hausdorff distances with the simpler Hausdorff distances.

We begin in Section~\ref{sec:background} with background on Hausdorff distances, Gromov--Hausdorff distances, correspondences, \v{C}ech complexes, the nerve lemma, the convexity radius of a manifold, and the fundamental class of a manifold.
In Section~\ref{sec:manifold-simple} we prove Theorem~\ref{thm:main1}(a), namely that $d_\gh(X,M)\geq\frac{1}{2}d_\h(X,M)$ for $X$ a sufficiently dense subset of a manifold $M$.
We introduce Vietoris--Rips complexes, which we use to show the constant $\frac{1}{2}$ can be improved to $1$ in the case of the circle (Theorem~\ref{thm:main2}), in Section~\ref{sec:circle}.
Section~\ref{sec:manifold-simple-two} is dedicated to the proof of Theorem~\ref{thm:main1}(b) for two subsets $X,Y\subseteq M$.
In Section~\ref{sec:fill-rad} we define the filling radius of a manifold to prove Theorem~\ref{thm:main3} with a relaxed density condition.
In Section~\ref{sec:Jung} we provide background on Hausmann's Theorem and Jung's Theorem,
which we use to prove our improved bound for a general manifold (Theorem~\ref{thm:main4}) with constant $\alpha(n,\kappa)$ in front of the $d_\h(X,M)$ term. 
We show in Section~\ref{sec:ratio} that for general metric space subsets the ratio between the Hausdorff and Gromov--Hausdorff distances can be arbitrarily far away from one (Theorem~\ref{thm:main5}).
In Section~\ref{sec:conclusion} we conclude with a discussion of how this project began, and a list of open questions that we hope are of interest to others.

\section{Background and notation}
\label{sec:background}

We present here the basic background material and notation used throughout the paper.
See the following textbooks for more details:~\cite{BuragoBuragoIvanov} for metric spaces and the Hausdorff and Gromov--Hausdorff distance;~\cite{Kozlov-book} for simplicial complexes and the nerve lemma;~\cite{hausmann2015mod} for $\Z/2$ homology and cohomology;~\cite{berger2007panoramic} for Riemannian manifolds and curvatures.

\subsection*{Metric spaces}
Let $(X,d)$ be a metric space.
For any $x\in X$, we let $B(x;r)=\{x'\in X~|~d(x,x') < r\}$ denote the metric ball of radius $r$ about $x$.
If needed for the sake of clarity, we will sometimes write $B_X(x;r)$ instead of $B(x;r)$.
For any $r\geq 0$ and $A\subseteq X$, we let $B(A;r)=\cup_{x\in A}B(x;r)$ denote the union of metric balls of radius $r$ centered at the points of $A$. 

The \emph{diameter} of a subset $A\subseteq X$ is the supremum of all distances between pairs of points of $A$.
More formally, the diameter is defined as
\[\diam(A)=\sup\{d(x,x') \mid x,x'\in A\}.\]
The diameter of a compact set $A$ is always finite.

\subsection*{Hausdorff distance}

Let $Z$ be a metric space.
If $X$ and $Y$ are two submetric spaces of $Z$ then the \emph{Hausdorff distance} between $X$ and $Y$ is
\[d^Z_\mathrm{H}(X,Y)=\inf\{r\geq 0\mid X\subseteq B(Y;r)\text{ and }Y\subseteq B(X;r)\}.\]
In other words, the Hausdorff distance finds the infimal real number such that if we thicken $Y$ by $r$ it contains $X$ and if we thicken $X$ by $r$ it contains $Y$; if no such $r$ exists then the Hausdorff distance is infinity.

\subsection*{Gromov--Hausdorff distance}

The Gromov--Hausdorff distance $d_{\gh}(X,Y)$ between two metric spaces $(X,d_X)$ and $(Y,d_Y)$ is defined as the infimum, over all metric spaces $Z$ and isometric embeddings $\gamma\colon X\to Z$ and $\varphi\colon Y\to Z$, of the Hausdorff distance in $Z$ between $\gamma(X)$ and $\varphi(Y)$~\cite{edwards1975structure,gromov1981groups, gromov1981structures,tuzhilin2016invented}.
To see that one can avoid taking an infimum over a proper class, one can restrict attention to the case where $Z$ is the disjoint union of $X$ and $Y$, equipped with a metric extending the metrics on $X$ and $Y$.
The Gromov--Hausdorff distance is a commonly used quantity in geometry and data analysis~\cite{GH-BU-VR,cheeger1997structure,colding1996large,gromov2007metric,harrison2023quantitative,lim2021gromov,petersen2006riemannian,tuzhilin2020lectures}.
Unlike the Hausdorff distance, the Gromov--Hausdorff distance considers sets $X$ and $Y$ that are not part of the same metric space.
However, to compute it we need to embed $X$ and $Y$ in different common metric spaces $Z$, and then take the infimum of the Hausdorff distance over all $Z$ and isometric embeddings.
Like the Hausdorff distance, it is possible for the Gromov--Hausdorff distance to be infinity.
The Gromov--Hausdorff distance indeed gives a metric on the isometry classes of compact metric spaces.

Note that for any two metric spaces $X$ and $Y$ contained in some common metric space $Z$, we always have
\begin{equation}
\label{eq:dgh-le-dh}
d_\gh(X,Y) \le d_\h^Z(X,Y).
\end{equation}

\subsection*{Correspondences and distortion}

A relation $R\subseteq X\times Y$ is a \emph{correspondence} if the following two conditions hold:
\begin{itemize}
\item For every $x \in X$, there exists some $y \in Y$ such that $(x,y) \in R$, and
\item For every $y \in Y$, there exists some $x \in X$ such that $(x,y) \in R$.
\end{itemize}
So every point in X is related to at least one point in Y, and vice-versa.
If $C$ is a correspondence between $(X,d_X)$ and $(Y,d_Y)$, then the \emph{distortion} of $C$ is:
\[\dis(C)=\sup_{(x,y)\in C,(x',y')\in C} \left|d_{X}(x,x')-d_{Y}(y,y')\right|.\]
We can equivalently define the \emph{Gromov--Hausdorff} distance between two metric spaces $X$ and $Y$ to be
\[d_{\gh}(X,Y) = \tfrac{1}{2} \inf_{C \subseteq X \times Y}\dis(C),\]
where the infimum is taken over all correspondences between $X$ and $Y$~\cite{BuragoBuragoIvanov,kalton1999distances}.

If $X$ and $Y$ are both compact metric spaces, then the infimum in the definition of the Gromov--Hausdorf distance is attained~\cite{ivanov2016realizations}.

\subsection*{Simplicial complexes.}

For notational convenience, we identify simplicial complexes with their geometric realizations.
A simplicial map $f\colon K\to L$ is continuous on geometric realizations.
If $K$ and $L$ are simplicial complexes and $f, g\colon K\to L$ are simplicial maps, then $f$ and $g$ are said to be \emph{contiguous} if for every simplex $\sigma\in K$, the union $f(\sigma)\cup g(\sigma)$ is contained in some simplex in $L$.
Contiguous simplicial maps induce homotopic maps on their geometric realizations.

\begin{figure}[htb]
\centering
\includegraphics[width=\textwidth]{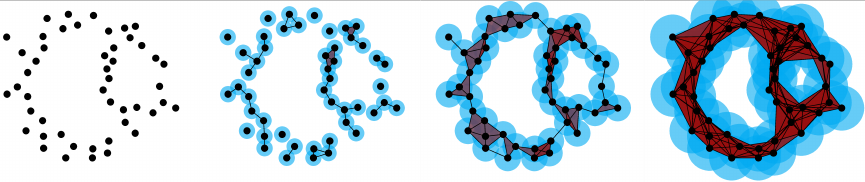}
\caption{A finite set $X$ of points in the plane, along with three \v{C}ech complexes $\cech{X}{r}$ with increasing values of $r>0$.}
\label{fig:cech}
\end{figure}

For $Z$ a metric space, $X\subseteq Z$, and $r>0$, the \v{C}ech complex $\cech{X}{r}$ is defined as the simplicial complex having $X$ as its vertex set,
and a finite subset $[x_0,\ldots,x_k]\subseteq X$ as a simplex if there is some $z\in Z$ with $d(z,x_i)<r$ for all $0\leq i\leq k$, i.e.\ if $\cap_{i=0}^k B_Z(x_i;r)\neq \emptyset$.
See Figure~\ref{fig:cech} for an example.
We emphasize that this is the \emph{ambient} \v{C}ech complex, since we are looking at intersections of balls in $Z$.
The case when $Z=X$ is often referred to as the \emph{intrinsic} \v{C}ech complex, but this case is not used in our paper.

\subsection*{Discontinuous functions induce simplicial maps}

\v{C}ech complexes are related to the theory of Gromov--Hausdorff distances because they allow us to turn discontinuous functions between metric spaces into continuous maps between simplicial complexes.
Indeed, we refer the reader to~\cite{ChazalDeSilvaOudot2014} for more information on the following lemma.

\vspace{3mm}
\begin{lemma}[Correspondences and ambient \v{C}ech complexes~\cite{ChazalDeSilvaOudot2014}]
\label{lem:correspondence-cech}
Let $Z$ be a metric space, let $X\subseteq Z$, and let $r > 2\cdot d_\gh(X,Z)$.
Then for any $\varepsilon > 0$, there exist simplicial maps
\[\cech{Z}{\varepsilon} \xlongrightarrow{\oh} \cech{X}{r+\varepsilon}
\xlongrightarrow{\og} \cech{Z}{3r+2\varepsilon}\]
such that the composition
$\og\circ\oh$ is contiguous to the inclusion
$\iota:\cech{Z}{\varepsilon}\hookrightarrow\cech{Z}{3r+2\varepsilon}$, where all complexes are ambient \v{C}ech complexes using balls in $Z$.
\end{lemma}
\vspace{3mm}

\begin{proof}
Since $2 \cdot d_\gh(X,Z)<r$, there exists a correspondence $C\subseteq X\times Z$ such that $\dis(C)<r$.
We can define (possibly discontinuous) functions $h\colon Z\to X$ and $g\colon X\to Z$ such that $(h(z),z)\in C$ for any $z\in Z$ and $(x,g(x))\in C$ for any $x\in X$.
    
We first show that the map $h$ extends to a simplicial map
$\oh\colon\cech{Z}{\varepsilon}\to \cech{X}{r+\varepsilon}$ on ambient \v{C}ech complexes defined via $\oh([z_0,\ldots,z_k])=[h(z_0),\ldots,h(z_k)]$.
Indeed, $[z_0,\ldots,z_k]\in\cech{Z}{\varepsilon}$ implies that there exists some $z\in Z$ with $d_Z(z,z_i)<\varepsilon$ for $0\leq i\leq k$.
Since $(h(z),z),(h(z_i),z_i)\in C$, 
\[
    d_X(h(z), h(z_i))\leq d_Z(z,z_i)+r<r+\varepsilon.
\]
So, $d_X(h(z),h(z_i))<r+\varepsilon$ for all~$i$, meaning $\oh([z_0,\dots,z_k]) \in \cech{X}{r+\varepsilon}$.
So $\oh$ is a simplicial map.

We now show that $g$ also extends to a simplicial map $\og \colon \cech{X}{r+\varepsilon}\to\cech{Z}{3r+2\varepsilon}$.
Note $[x_0,\ldots,x_k]\in\cech{X}{r+\varepsilon}$ implies that there exists some $z\in Z$ (since we are using ambient \v{C}ech complexes in $Z$) with $d_Z(z,x_i)<r+\varepsilon$ for $0\leq i\leq k$.
From the triangle inequality, for any $0\leq i\leq k$ we get
\[d_Z(x_0,x_i)\leq d_Z(x_0,z) +d_Z(z,x_i)<(r+\varepsilon)+(r+\varepsilon)=2r+2\varepsilon.\]

Since $(x_0,g(x_0)),(x_i,g(x_i))\in C$, we therefore get
\[d_Z(g(x_0), g(x_i))\leq d_X(x_0,x_i)+r<(2r+2\varepsilon)+r=3r+2\varepsilon.\]
So, $d_Z(g(x_0),g(x_i))<3r+2\varepsilon$ for all~$i$, meaning $\og([x_0,\dots,x_k]) \in \cech{Z}{3r+2\varepsilon}$.

In order to show that the simplicial maps $\og\circ\oh$ and $\iota$ are contiguous, let $\sigma=[z_0,\ldots,z_k]\in\cech{Z}{\varepsilon}$.
So there exists $z\in Z$ with $d_Z(z,z_i)<\varepsilon$ for all $i$.
Since $(h(z),z),(h(z_i),g(h(z_i)))\in C$, by the definition of distortion we have
\[
    d_Z(z,g(h(z_i)))\leq d_X(h(z),h(z_i))+r
    \leq d_Z(z,z_i)+2r < 2r+\varepsilon\leq3r+2\varepsilon.
\]
So $B_Z(z;3r+2\varepsilon)$ contains $z_i$ and $g(h(z_i))$ for all $i$, meaning $(\og\circ\oh)(\sigma)\cup\iota(\sigma)\in\cech{Z}{3r+2\varepsilon}$.
\end{proof}

\subsection*{Riemannian manifolds}
Let $M$ be a Riemannian manifold without boundary.
A subset $A\subseteq M$ is called \emph{geodesically convex} if for any two points $p,q\in A$, there exists a unique minimizing geodesic segment from $p$ to $q$, and furthermore this geodesic is entirely in $A$.

\vspace{3mm}
\begin{definition}[Convexity radius]
\label{def:conv-rad}
The \emph{convexity radius} of $M$, denoted $\rho(M)$, is the supremum of the set of radii so that all smaller balls are geodesically convex.
Formally,
\[
\rho(M)=\sup\;\{r\geq0\mid B(p;s)
\text{ is geodesically convex for all }0<s<r\text{ and }p\in M\}.	
\]
\end{definition}

If $M$ is compact, the convexity radius $\rho(M)$ is positive; see~\cite[Proposition~95]{berger2007panoramic}.
For example, the unit sphere $S^n$ has
convexity radius $\rho(S^n)=\tfrac{\pi}{2}$.

The main theorems in this paper will be for a connected, closed Riemannian manifold $M$, where \emph{closed} means compact and without boundary.

\subsection*{Nerve lemma}

An open cover $\U=\{U_i\}_{i\in I}$ of a topological space $Z$ is a \emph{good cover} if each set is contractible and all finite intersections of its elements are empty or contractible.
The \emph{nerve} of~$\U$, denoted $\N(\U)$, is defined to be the simplicial
complex having $I$ as its vertex set, and a simplex $\{i_0,\ldots,i_k\}$ when the finite intersection $U_{i_0}\cap\ldots\cap U_{i_k}$ is nonempty (see Figure~\ref{fig:nerve}).
Under the right assumptions,
the nerve preserves the homotopy~type of the union $Z$, as stated by the
following foundational result.

\begin{figure}[htb]
\centering
\includegraphics[width=5in]{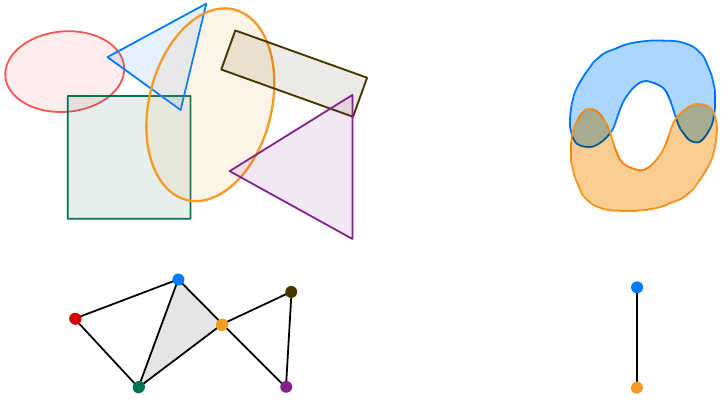}
\caption{\emph{(Left)} A good cover of a topological space and its nerve.
\emph{(Right)} A cover of a topological space that is not good, along with its nerve.}
\label{fig:nerve}
\end{figure}

\vspace{3mm}
\begin{lemma}[Nerve lemma~\cite{Alexandroff1928,Borsuk1948,Hatcher}]
\label{lem:nerve}
If $\U=\{U_i\}_{i\in I}$ is a good open cover of a paracompact topological space $Z$, then we have a homotopy equivalence $\N(\U)\simeq Z$.

Furthermore, if $\U'=\{U'_i\}_{i\in I}$ is another good open cover of the paracompact space $Z$ with $U_i\subseteq U'_i$ for all $i\in I$, then the natural inclusion $\iota\colon \N(\U)\hookrightarrow\N(\U')$ is a homotopy equivalence.
\end{lemma}
\vspace{3mm}

We refer to this latter statement as a \emph{functorial} nerve lemma.
It is implied by the standard proof of the nerve lemma as follows.
When defining the map from $Z\to\N(\U')$, use the same partition of unity subordinate to $\U$ used to obtain the map $Z\to\N(\U)$.
This produces the commutative diagram below, which implies that the natural inclusion $\iota$ is a homotopy equivalence.
See for example~\cite[Lemma 3.7]{fasy2022geodesic}.
\begin{equation*}
\begin{tikzpicture} [baseline=(current  bounding  box.center)]
\centering
\node (k1) at (-2,0) {$\N(\U)$};
\node (k2) at (2,0) {$\N(\U')$};
\node (k3) at (0,-1.5) {$Z$};
\draw[rinclusion] (k1) to node[auto] {$\iota$} (k2);
\draw[map] (k3) to node[auto] {$h$} (k1);
\draw[map] (k3) to node[auto,swap] {$h'$} (k2);
\end{tikzpicture}        
\end{equation*}
See~\cite{bauer2023nerve} for further variants of the functorial nerve lemma, for example in the setting when the cover need not be good.

We note that if $X$ is a metric space and $A\subseteq X$, then the \v{C}ech complex is an example of a nerve complex, namely $\cech{A}{r}=\N(\{B_X(a;r)~|~a\in A\})$.

\vspace{3mm}
\begin{corollary}
\label{cor:cech-nerve}
Let $M$ be a Riemannian manifold.
Then $\cech{M}{r} \simeq M$ for all $r<\rho(M)$.
Furthermore, if $0<r\leq r'<\rho(M)$, then the natural inclusion $\iota\colon \cech{M}{r}\hookrightarrow\cech{M}{r'}$ is a homotopy equivalence.
\end{corollary}

\begin{proof}
Since $r < \rho(M)$, the ball $B(m;r)$ is geodesically convex for all $m\in M$.
Furthermore, the open cover $\{B(m;r)\}_{m\in M}$ of $M$ is good, since an intersection of geodesically convex sets is geodesically convex, and since a nonempty geodesically convex set is contractible.
And $M$ is paracompact since it is a metric space.
Therefore, the Nerve Lemma (Lemma~\ref{lem:nerve}) implies $\cech{M}{r}\simeq M$.

Lemma~\ref{lem:nerve} also implies that the inclusion map is a homotopy equivalence.
\end{proof}

We note the cover $\{B(m;r)\}_{m\in M}$ of $M$ is not locally finite.
Nevertheless, since $M$ is paracompact, a (locally finite) partition of unity exists subordinate to this open cover.
This idea is explained in~\cite[Theorem~13.1.3]{tom2008algebraic}, and used in the proof of~\cite[Corollary~4G.3]{Hatcher}.

\subsection*{The fundamental class of a manifold}
In this paper, we always take homology with $\Z/2$-coefficients and manifolds
without boundary. 
Let $M$ be a connected (not necessarily compact) $n$--dimensional manifold without boundary. 
The $n$-dimensional homology $H_n(M;\Z/2)$ of $M$ is well-understood and is dictated by the compactness of $M$.
More specifically, 
\begin{equation}\label{eq:fund-cls}
    H_n(M;\Z/2)\simeq\begin{cases} 
    \Z/2&\text{ if }M\text{ is compact} \\
    0&\text{ if }M\text{ is not compact}.
    \end{cases}
\end{equation}
To see this, we observe from the Poincar\'e duality that 
\[
H_n(M;\Z/2)\simeq H^0_c(M;\Z/2),
\]
where $H^0_c(M)$ denotes the $0$-dimensional cohomology group with \emph{compact support}. 
If $M$ is compact, cohomology with compact support matches the usual cohomology, and so $H^0_c(M,\Z/2)=\Z/2$. 
The \emph{fundamental class} over $\Z/2$ of a compact, connected $n$--manifold $M$ is the (non-trivial) generator $[M]$ of the homology group $H_n(M,\Z/2)$; see~\cite{hausmann2015mod}. 
In the case when $M$ is not compact, one can see that a $0$-cocycle can not be supported on a compact subset of $M$, i.e., $H^0_c(M,\Z/2)=0$. 

\vspace{3mm}
\begin{lemma}
\label{lem:proper-subset}
Let $M$ be a connected closed $n$-manifold.
If $M'$ is an open proper subset of $M$, then $H_n(M') \cong 0$. 
\end{lemma}

\begin{proof}
The lemma is clearly true if $M'$ is empty, so let $M'\neq\emptyset$.
The space $M'$ is an $n$-manifold since it is an open subset of $M$.
We further note that $M'$ is not compact, since then it would be a nonempty proper subset that is both open and closed in $M$, contradicting the fact that $M$ is connected.
The result then follows from \eqref{eq:fund-cls}.
\end{proof}

\section{A bound in manifolds, with a non-optimal constant}
\label{sec:manifold-simple}

We recall the statement of Theorem~\ref{thm:main1}.

\vspace{3mm}
\mainOne*
\vspace{3mm}

The purpose of this section is to give the proof of part (a).
We defer the proof of part (b) until Section~\ref{sec:manifold-simple-two}.

\begin{proof}[Proof of Theorem~\ref{thm:main1}(a)]
We will show that if $d_\gh(X,M) < \frac{1}{6}\rho(M)$, then $d_\gh(X,M) \ge \frac{1}{2}d_\h(X,M)$.

Fix $r$ such that $2\cdot d_\gh(X,M)<r
<\frac{1}{3}\rho(M)$, and fix $\varepsilon>0$ so that $3r+2\varepsilon<\rho(M)$.
Every \v{C}ech complex in this proof is an ambient \v{C}ech complex with balls in $M$.
By Lemma~\ref{lem:correspondence-cech} we get simplicial maps
\[
\cech{M}{\varepsilon} 
\xlongrightarrow{\oh} \cech{X}{r+\varepsilon}
\xlongrightarrow{\og} \cech{M}{3r+2\varepsilon}
\]    
whose composition $\og \circ \oh$ is homotopic to the inclusion
$\iota\colon \cech{M}{\varepsilon} \hookrightarrow \cech{M}{3r+2\varepsilon}$.
The two collections of balls $\{B(m;\varepsilon)\}_{m\in M}$ and $\{B(m;3r+2\varepsilon)\}_{m\in M}$ each form good open covers of $M$, since $0<\varepsilon<3r+2\varepsilon<\rho(M)$.
The nerve lemma (Corollary~\ref{cor:cech-nerve}) therefore implies that $\og \circ \oh$ is a homotopy equivalence between spaces homotopy equivalent to $M$.
Hence $\og_*\circ\oh_*$ preserves the fundamental class, i.e.\ $\og_*\circ\oh_*$ is an isomorphism on $n$-dimensional homology with $\Z/2$ coefficients, where $n$ is the dimension of the manifold $M$:
\[\Z/2 \cong H_n(\cech{M}{\varepsilon}) \xrightarrow{\oh_*} H_n(\cech{X}{r+\varepsilon}) \xrightarrow{\og_*} H_n(\cech{M}{3r+2\varepsilon}) \cong \Z/2.\]

\begin{figure}[htb]
\centering
\includegraphics[width=3in]{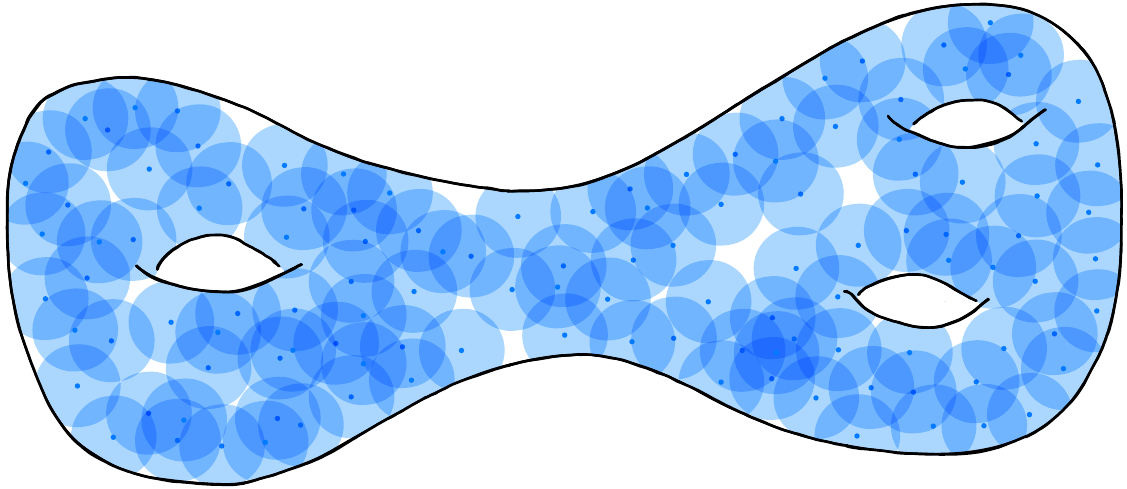}
\caption{Picture of the balls determining $\cech{X}{r+\varepsilon}$ with $r+\varepsilon < d_\h(X,M)$ in the proof of Theorem~\ref{thm:main1}(a); note that the balls do not cover $M$.}
\label{fig:proof-main-thm}
\end{figure}

If we had $r+\varepsilon < d_\h(X,M)$, then the collection $\{B(x,r+\varepsilon)\}_{x\in X}$ of open balls would not cover $M$; see Figure~\ref{fig:proof-main-thm}.
By the nerve lemma, $\cech{X}{r+\varepsilon}$ would be homotopy equivalent to an open proper subset of $M$, implying $H_n(\cech{X}{r+\varepsilon})=0$ by Lemma~\ref{lem:proper-subset}.
This would contradict the fact that $(\og \circ \oh)_*$ is nonzero on $H_n$.
Therefore, it must be that $r+\varepsilon \ge d_\h(X,M)$ for all $2\cdot d_\gh(X,M)<r
<\frac{1}{3}\rho(M)$ and sufficiently small $\varepsilon>0$.
Thus $2\cdot d_\gh(X,M) \ge d_\h(X,M)$, i.e.\ $d_\gh(X,M)\geq\frac{1}{2}d_\h(X,M)$, as desired.
\end{proof}

\section{Vietoris--Rips complexes and the case of the circle}
\label{sec:circle}

We show how to improve the constant on the term $\frac{1}{2}d_\h(X,M)$ in the lower bounds in Theorem~\ref{thm:main1}.
We first do so in this section in the case when $M=S^1$ is the circle, when we improve the constant $\frac{1}{2}$ all the way to $1$, which cannot be improved further.
Our main tool, Vietoris--Rips complexes, will also be important when improving the constant $\frac{1}{2}$ in the case of general manifolds $M$ in Section~\ref{sec:Jung}.

\subsection*{Vietoris--Rips complexes}
For a metric space $(X,d)$ and scale $r > 0$, the \emph{Vietoris--Rips} complex $\vr{X}{r}$ is defined as the simplicial complex having $X$ as its vertex set,
and each finite subset $\sigma$ of $X$ with $\diam(\sigma)<r$ as a simplex.
In other words, $[x_0,\ldots,x_k]$ is a simplex if $d(x_i,x_j)<r$ for all $0\le i,j\le k$.
See Figure~\ref{fig:vr}.

\begin{figure}[htb]
\centering
\includegraphics[width=1.1in]{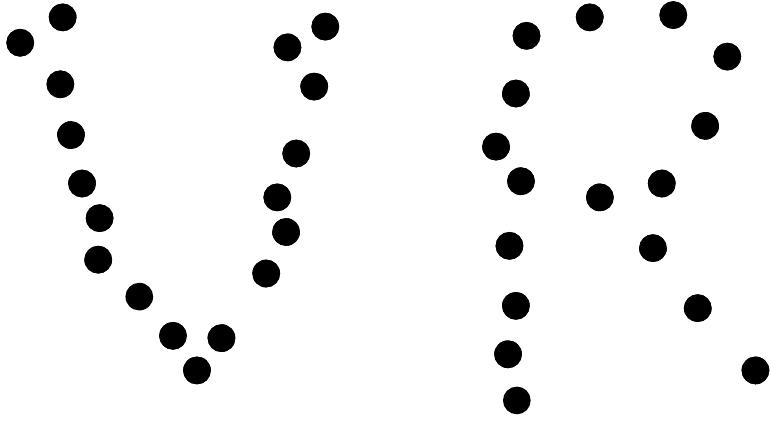}
\hspace{0.17in}
\includegraphics[width=1.1in]{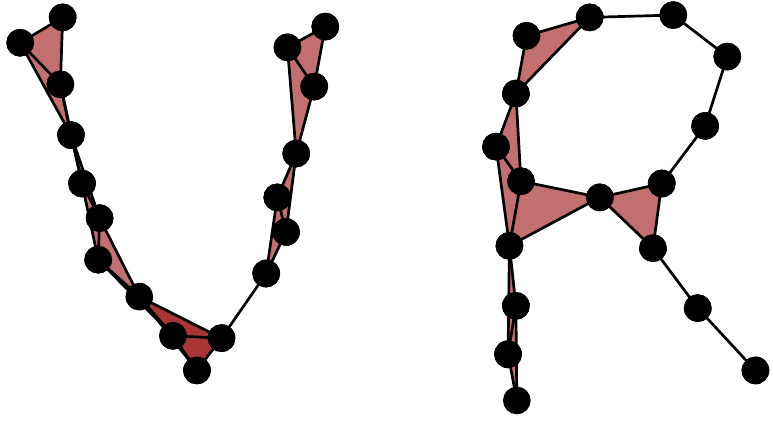}
\hspace{0.17in}
\includegraphics[width=1.1in]{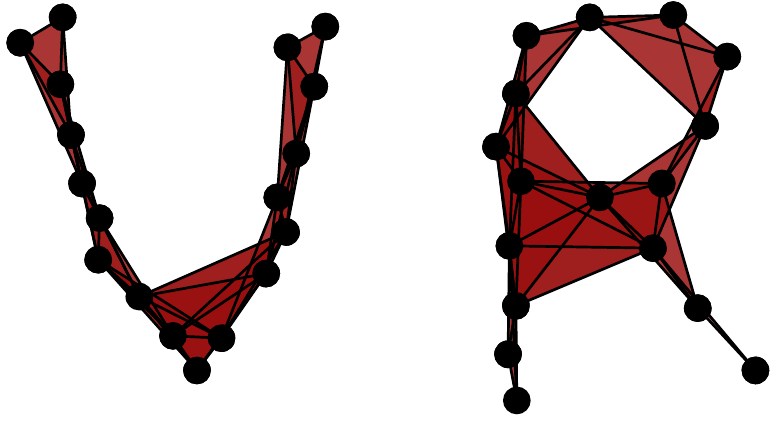}
\hspace{0.17in}
\includegraphics[width=1.1in]{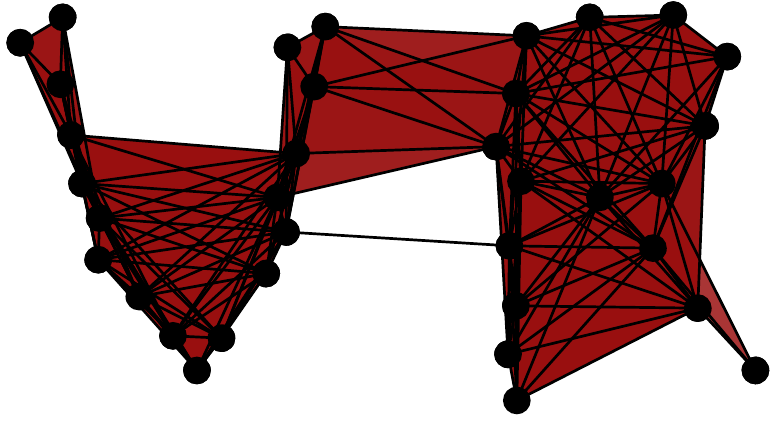}
\caption{A finite set $X$ of points in the plane, along with three Vietoris--Rips complexes $\vr{X}{r}$ with increasing values of $r>0$.}
\label{fig:vr}
\end{figure}

Vietoris--Rips complexes are related to the theory of Gromov--Hausdorff
distances because they allow us to turn discontinuous functions between metric spaces into continuous maps between simplicial complexes; this perspective from~\cite{ChazalDeSilvaOudot2014} was also recently used in~\cite{GH-BU-VR}.

\vspace{3mm}
\begin{lemma}[Correspondences and Vietoris--Rips complexes~\cite{ChazalDeSilvaOudot2014}]
\label{lem:correspondence-vr}
Let $X$ and $Y$ be metric spaces and let $r > 2\cdot d_\gh(X,Y)$.
Then for any $\varepsilon > 0$ and $\nu\geq0$, there exist simplicial maps
\[\vr{Y}{\varepsilon} \xlongrightarrow{\oh} \vr{X}{r+\varepsilon} \xhookrightarrow{\iota_\nu} \vr{X}{r+\varepsilon+\nu}
\xlongrightarrow{\og} \vr{Y}{2r+\varepsilon+\nu}\] 
such that the composition $\og\circ\iota_\nu\circ\oh$ is contiguous to the inclusion $\vr{Y}{\varepsilon}\hookrightarrow\vr{Y}{2r+\varepsilon+\nu}$.
\end{lemma}
\vspace{3mm}

We defer the proof to Appendix~\ref{app:proofs-interleavings}, since it is analogous to the proof of Lemma~\ref{lem:correspondence-cech}.

\begin{remark}
We remark that frequently one chooses $\nu=0$ in Lemma~\ref{lem:correspondence-vr}, in which case the inclusion map $\iota_\nu$ in the middle disappears.
This gives
\[\vr{Y}{\varepsilon} \xlongrightarrow{\oh} \vr{X}{r+\varepsilon} \xlongrightarrow{\og} \vr{Y}{2r+\varepsilon},\]
with the composition $\og\circ\oh$ contiguous to the inclusion $\iota\colon\vr{Y}{\varepsilon}\hookrightarrow\vr{Y}{2r+\varepsilon}$.
\end{remark}

Later sections will also require the following lemma, which gives a tighter interleaving in the case of subset inclusion.

\vspace{3mm}
\begin{lemma}[Hausdorff distance and Vietoris--Rips complexes]
\label{lem:rips-hausdorff}
Let $Z$ be a metric space, and let $Y\subseteq Z$ satisfy $r>2\cdot d_\h(Y,Z)$.
Then for any $\varepsilon > 0$ and $\nu \ge 0$, there exist simplicial maps
\[\vr{Z}{\varepsilon} \xlongrightarrow{\of} \vr{Y}{r+\varepsilon} \xhookrightarrow{\iota_\nu} \vr{Y}{r+\varepsilon+\nu}
\xhookrightarrow{\iota} \vr{Z}{r+\varepsilon+\nu}\]
such that the composition $\iota \circ \iota_\nu \circ \of$ is contiguous to the inclusion $\vr{Z}{\varepsilon}\hookrightarrow\vr{Z}{r+\varepsilon+\nu}$.
\end{lemma}
\vspace{3mm}

We again defer the proof to Appendix~\ref{app:proofs-interleavings}.

\subsection*{The case of the circle}

Let $S^1$ be the circle with the geodesic metric and circumference $2\pi$.
In Theorem~\ref{thm:main2}, we use Vietoris--Rips complexes to improve the bound from Theorem~\ref{thm:main1} in the specific case of the circle. 
Indeed, we show how to improve the bound $d_\gh(X,S^1) \ge \frac{1}{2} \cdot d_\h(X,S^1)$ from Theorem~\ref{thm:main1} to be an \emph{equality}, namely $d_\gh(X,S^1) = d_\h(X,S^1)$, for $X\subseteq S^1$ sufficiently dense.

First, we prove a lemma demonstrating the equivalence of \v{C}ech and Vietoris--Rips complexes at small scales. 

\vspace{3mm}
\begin{lemma}[\v{C}ech--Rips equivalence]
\label{lem:circle-cech}
For any subset $X\subseteq S^1$ and $0<r<\frac{2\pi}{3}$, we have the simplicial inclusion $\vr{X}{r}\hookrightarrow\cech{X}{\frac{r}{2}}$. 
As a consequence, $\vr{X}{r}=\cech{X}{\frac{r}{2}}$.
\end{lemma}

\begin{proof}
Let $\sigma=[x_0,\ldots,x_k]\in\vr{X}{r}$, which means $\diam(\sigma) < r$. 
Without any loss of generality, we assume that the geodesic distance between $x_0$ and $x_k$ equals $\diam(\sigma)$.
Since $r<\frac{2\pi}{3}$, this implies that $x_0$ and $x_k$ are the \emph{unique} points in $\sigma$ with $d(x_0,x_k)=\diam(\sigma)$, and furthermore that $\sigma$ is contained in the unique geodesic joining $x_0$ and $x_k$.
If $y\in S^1$ is the midpoint of this geodesic arc, then the metric ball 
of radius $\frac{r}{2}$ centered at $y$ covers $\sigma$, giving $\sigma\in\cech{X}{\frac{r}{2}}$.
This implies that $\vr{X}{r}\hookrightarrow\cech{X}{\frac{r}{2}}$.
It is always true that $\cech{X}{\frac{r}{2}}\hookrightarrow\vr{X}{r}$, due to the triangle inequality, giving the other direction of the equality $\vr{X}{r}=\cech{X}{\frac{r}{2}}$.
\end{proof}

This lemma will be generalized to the setting of a general manifold in Lemma~\ref{lem:rips-cech}, where we relate it to Jung's theorem~\cite{danzer1963helly}.
In the proof of Lemma~\ref{lem:circle-cech}, the fact that $\sigma$ is contained in the geodesic joining $x_0$ and $x_k$ reminds us of a local version of the definition of $\delta$-hyperbolicity~\cite{bridson2011metric} (small enough triangles in the circle are $\delta$-slim with $\delta=0$), and also of a local version of Helly's theorem~\cite{danzer1963helly,de2019discrete} (small enough arcs in the circle that intersect pairwise have a point of common intersection).

\vspace{3mm}
\mainTwo*

\begin{proof}[Proof of Theorem~\ref{thm:main2}]
Part (a) of Theorem~\ref{thm:main2} is obtained from part (b) by letting $Y=S^1$.
Hence it suffices to prove (b).

Let $\delta=d_\h(Y,S^1)$.
We will show if $d_\gh(X,Y)<\frac{\pi}{6}-\frac{\delta}{2}$, then $d_\gh(X,Y)\geq d_\h(X,S^1)-\delta$.
    
Fix $r$ such that $2\cdot d_\gh(X,Y) < r < \frac{\pi}{3}-\delta$, and fix $\varepsilon>0$ so that $r+\delta+\varepsilon < \frac{\pi}{3}$.
By Lemma~\ref{lem:correspondence-vr}, we get continuous maps
\[\vr{Y}{2\delta+2\varepsilon} \xlongrightarrow{\oh} \vr{X}{r+2\delta+2\varepsilon} \xlongrightarrow{\og} \vr{Y}{2r+2\delta+2\varepsilon}\]
whose composition $\og\circ\oh$ is homotopic to the inclusion $\iota:\vr{Y}{2\delta+2\varepsilon}\hookrightarrow\vr{Y}{2r+2\delta+2\varepsilon}$.
Since $2r+2\delta+2\varepsilon < \frac{2\pi}{3}$, using Lemma~\ref{lem:circle-cech} we can write these complexes as
\[\cech{Y}{\delta+\varepsilon} \xlongrightarrow{\oh} \cech{X}{\tfrac{r}{2}+\delta+\varepsilon} \xlongrightarrow{\og} \cech{Y}{r+\delta+\varepsilon}.\]
The two collections of balls $\{B(y;\delta+\varepsilon)\}_{y\in Y}$ and $\{B(y;r+\delta+\varepsilon)\}_{y\in Y}$ each form good open covers of $S^1$, since $d_\h(Y,S^1)<\delta+\varepsilon< r+\delta+\varepsilon<\frac{\pi}{3}<\frac{\pi}{2}=\rho(S^1)$.
The nerve lemma (Corollary~\ref{cor:cech-nerve}) therefore implies that $\iota$ and $\og \circ \oh$ are each homotopy equivalences between spaces homotopy equivalent to $S^1$.
So $\og_*\circ\oh_*$ is an isomorphism on $1$-dimensional homology.

If we had $\frac{r}{2}+\delta+\varepsilon<d_\h(X,S^1)$, then the collection of open balls $\left\{B\left(x,\frac{r}{2}+\delta+\varepsilon\right)\right\}_{x\in X}$ would not cover $S^1$.
By the nerve lemma, $\cech{X}{\frac{r}{2}+\delta+\varepsilon}$ would be homotopy equivalent to an open proper subset of $S^1$, implying $H_1\left(\cech{X}{\frac{r}{2}+\delta+\varepsilon}\right)=0$ by Lemma~\ref{lem:proper-subset}.
This would contradict the fact that $(\og\circ\oh)_*$ is nonzero on $H_1(\bullet)$.
Therefore, it must be that $\frac{r}{2}+\delta+\varepsilon \ge d_\h(X,S^1)$ for all $2\cdot d_\gh(X,Y) < r < \frac{\pi}{3}-\delta$ and sufficiently small $\varepsilon>0$.
So $d_\gh(X,Y)+\delta \ge d_\h(X,S^1)$, giving $d_\gh(X,Y) \geq d_\h(X,S^1)-d_\gh(Y,S^1)$, as desired.
\end{proof}

Instead of passing to \v{C}ech complexes, the above proof could have instead used only Vietoris--Rips complexes, using the \emph{winding fraction}~\cite{AA-VRS1,AAM,AAR} to determine when the Vietoris--Rips complex of a subset of the circle is homotopy equivalent to the circle or not, and the fact that $\vr{S^1}{r}\simeq S^1$ for $r<\frac{2\pi}{3}$.
This is how our first proof of Theorem~\ref{thm:main2}(a), the first result we proved during this research project, proceeded; see Appendix~\ref{app:additional} for this proof.





\section{A bound for two subsets of a manifold using Hausmann's theorem}\label{sec:manifold-simple-two}

Before producing a lower bound on the Gromov--Hausdorff distance between two subsets of a general manifold, we review Hausmann's theorem.
The basic idea is that Hausmann's theorem is an analogue of the nerve lemma for Vietoris--Rips complexes instead of \v{C}ech complexes.

\subsection*{Hausmann's theorem}

The following result by Hausmann states that $M$ is homotopy equivalent to its Vietoris--Rips complex for a scale smaller than the convexity radius.

\vspace{3mm}
\begin{theorem}[Hausmann's Theorem~\cite{Hausmann1995}]
\label{thm:hausmann}
For any $0<r<\rho(M)$, we have a homotopy equivalence $\vr{M}{r} \simeq M$.

Furthermore, if $0<r\le r'<\rho(M)$, then the inclusion $\vr{M}{r}\hookrightarrow\vr{M}{r'}$ is a homotopy equivalence~\cite{virk2021rips}.
\end{theorem}
\vspace{3mm}

See~\cite{AAF} for an alternate proof of Hausmann's theorem in the setting of Vietoris--Rips metric thickenings, which (with the $<$ convention) are now known to be weakly homotopy equivalent to the Vietoris--Rips simplicial complexes~\cite{gillespie2023vietoris,gillespie2022homological,HA-FF-ZV}.

\subsection*{Two subsets of a manifold}

We now prove Theorem~\ref{thm:main1}(b) for two subsets of a manifold, which states the following.
Let  $M$ be a connected, closed Riemannian manifold with convexity radius $\rho(M)$.
Then for any $X,Y\subseteq M$, we have
\[d_\gh(X,Y) \ge \min\left\{\tfrac{1}{2}d_\h(X,M)-d_\h(Y,M), \tfrac{1}{6}\rho(M)-\tfrac{2}{3}d_\h(Y,M)\right\}.\]

\begin{proof}[Proof of Theorem~\ref{thm:main1}(b)]
We will show if $d_\gh(X,Y)<\frac{1}{6}\rho(M)-\frac{2}{3}d_\h(Y,M)$, then $d_\gh(X,Y)\geq \frac{1}{2}d_\h(X,M)-\frac{\delta}{2}$ for any $2\cdot d_\h(Y,M)<\delta<\frac{1}{2}\rho(M)-3\cdot d_\gh(X,Y)$.
Rearranging and multiplying by $\frac{2}{3}$, we get
\[2\cdot d_\gh(X, Y) < \tfrac{1}{3}\rho(M)-\tfrac{2}{3}\delta.\]
So, we can fix $r$ such that $2\cdot d_\gh(X,Y) < r < \frac{1}{3}\rho(M)-\frac{2}{3}\delta$, and fix $\varepsilon>0$ so that $3r+2\delta+2\varepsilon < \rho(M)$.
We produce the following sequence of continuous maps.
The natural inclusion of the Vietoris--Rips complex into the \v{C}ech complex at the same scale gives the inclusion $\iota$.
The inclusion $\iota'$ always holds for \v{C}ech and Vietoris--Rips complexes after doubling the scale.
We note that $\iota'\circ\iota$ is equal to an inclusion (dashed), i.e., the middle triangle commutes.
Since $2\cdot d_\gh(X,Y) < r$, we use Lemma~\ref{lem:correspondence-vr} to get the maps $\oh$ and $\og$, which satisfy that $\og\circ\iota'\circ\iota\circ\oh$ is homotopic to the inclusion map.
\begin{equation*}
\begin{tikzpicture} [scale= 0.9, baseline=(current  bounding  box.center)]
\centering
\node (k1) at (-5.4,-1.5) {$\vr{M}{\varepsilon}$};
\node (k2) at (-5.4,0) {$\vr{Y}{\delta+\varepsilon}$};
\node (k3) at (-2.1,0) {$\vr{X}{r+\delta+\varepsilon}$};
\node (k4) at (0,-1.5) {$\cech{X}{r+\delta+\varepsilon}$};
\node (k5) at (2.1,0) {$\vr{X}{2r+2\delta+2\varepsilon}$};
\node (k6) at (6.3,0) {$\vr{Y}{3r+2\delta+2\varepsilon}$};
\node (k7) at (6.3,-1.5) {$\vr{M}{3r+2\delta+2\varepsilon}$};
\draw[map] (k1) to node[auto] {$\of$} (k2);
\draw[map] (k2) to node[auto] {$\oh$} (k3);
\draw[rinclusion] (k3) to node[auto] {$\iota$} (k4);
\draw[rinclusion] (k4) to node[auto] {$\iota'$} (k5);
\draw[map] (k5) to node[auto] {$\og$} (k6);
\draw[rinclusion] (k6) to node[auto] {$\iota''$} (k7);
\draw[rinclusion, dashed] (k3) to (k5);
\end{tikzpicture}        
\end{equation*}
Finally, since $\delta>2\cdot d_\h(Y,M)$, Lemma~\ref{lem:rips-hausdorff} gives the maps $\of$ and $\iota''$, which satisfy that $\iota''\circ\og\circ\iota'\circ\iota\circ\oh\circ\of$ is homotopic to the inclusion $\vr{M}{\varepsilon}\hookrightarrow\vr{M}{3r+2\delta+2\varepsilon}$. 
By Hausmann's theorem (Theorem~\ref{thm:hausmann}) and the fact that $\varepsilon < 3r+2\delta+2\varepsilon < \rho(M)$, this inclusion is a homotopy equivalence between spaces homotopy equivalent to $M$.
Thus the composition $\iota''\circ\og\circ\iota'\circ\iota\circ\oh\circ\of$ preserves the nonzero fundamental class in $n$-dimensional homology, where $n$ is the dimension of the manifold~$M$.

If we had $r+\delta+\varepsilon < d_\h(X,M)$, then the collection $\left\{B\left(x,r+\delta+\varepsilon\right)\right\}_{x\in X}$ of open balls would not cover $M$.
By the nerve lemma, $\cech{X}{r+\delta+\varepsilon}$ would be homotopy equivalent to an open proper subset of $M$, implying $H_n\left(\cech{X}{r+\delta+\varepsilon}\right)=0$ by Lemma~\ref{lem:proper-subset}, a contradiction.
Therefore, it must be that $r+\delta+\varepsilon \ge d_\h(X,M)$ for all $2\cdot d_\gh(X,Y) < r < \frac{1}{3}\rho(M)-\frac{2}{3}\delta$ with $\varepsilon>0$ sufficiently small.
So, $2 \cdot d_\gh(X,Y)+\delta \ge d_\h(X,M)$, giving $d_\gh(X,Y) \ge \frac{1}{2}d_\h(X,M)-\frac{\delta}{2} \ge \frac{1}{2}d_\h(X,M)-d_\h(Y,M)$.
\end{proof}

Instead of considering Vietoris--Rips complexes of the manifold $M$ in the first and last steps of the sequence of maps in the proof of Theorem~\ref{thm:main1}(b), it might be possible to instead apply Latschev's theorem~\cite{Latschev2001,lemevz2022finite,majhi2023demystifying}.

\begin{figure}[htb]
\centering
\includegraphics[width=5in]{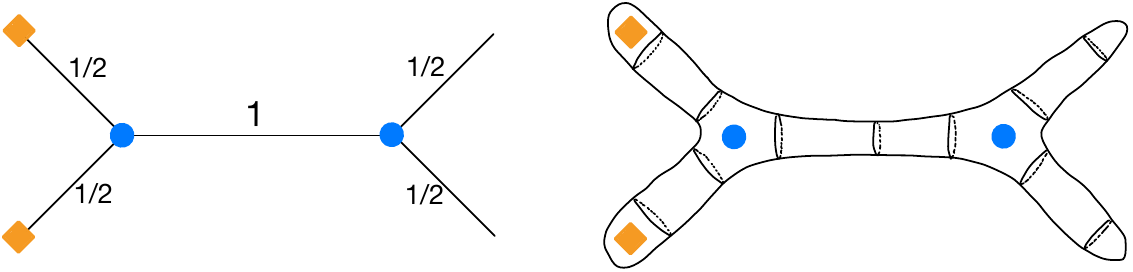}
\caption{
In both images, the subset $X$ (orange diamonds) consists of two points at distance $1$ apart, and same for the subset $Y$ (blue circles).
\emph{(Left)} A tree with four edges of length $\frac{1}{2}$ and one edge of length $1$.
Note $d_\h(X,M)=2=4\cdot\frac{1}{2}=4\cdot d_\h(Y,M)$.
\emph{(Right)} Riemannian surface $M$ is the boundary of the an $\varepsilon$-thickening of the tree in $\R^3$, equipped with the geodesic metric.
Note $d_\h(X,M)\approx 4\cdot d_\h(Y,M)$.
}
\label{fig:isometry-X-Y}
\end{figure}

\begin{remark}
A particular case of Theorem~\ref{thm:main1}(b) says the following.
If $X$ and $Y$ are two \emph{isometric} subsets of a connected, closed Riemannian manifold $M$, and if $Y$ is close enough to $M$ (meaning $d_\h(Y,M)<\frac{1}{4}\rho(M)$), then the Hausdorff distance from $X$ to $M$ and from $Y$ to $M$ cannot differ by more than a factor of two: $d_\h(X,M)\le 2\cdot d_\h(Y,M)$.
Indeed, if $X$ and $Y$ are isometric, then $d_\gh(X,Y)=0$, and $d_\h(Y,M)<\frac{1}{4}\rho(M)$ implies that the second term in the minimum on the right-hand side of Theorem~\ref{thm:main1}(b) is positive. 
This means that the first term in the minimum must be nonpositive, yielding $d_\h(X,M)\le 2\cdot d_\h(Y,M)$.
If $X$ and $Y$ are singletons, then the triangle inequality immediately gives $d_\h(X,M)\le 2\cdot d_\h(Y,M)$, without any need for the density assumption $d_\h(Y,M)<\frac{1}{4}\rho(M)$.
However, once the isometric subsets $X$ and $Y$ consist of at least two points, the density assumption $d_\h(Y,M)<\frac{1}{4}\rho(M)$ is needed in general.
Indeed, see the example in Figure~\ref{fig:isometry-X-Y}(right), in which $X$ and $Y$ each consist of two points in $M$ at unit distance apart, in which $Y$ is not dense enough compared to the convexity radius of $M$, and in which $d_\h(X,M)\approx 4\cdot d_\h(Y,M)$.
Of course, the full power of Theorem~\ref{thm:main1}(b) is that it can provide lower bounds on $d_\gh(X,Y)$ even when $X$ and $Y$ are not isometric; Figure~\ref{fig:isometry-X-Y} simply shows the necessity of a density assumption.    
\end{remark}

We conclude the section by presenting an immediate but useful corollary of Theorem~\ref{thm:main1}(b). 

\begin{corollary}\label{cor:two-subsets-eps}
Let $M$ be a connected, closed Riemannian manifold with convexity radius $\rho(M)$.
Then, for any $X, Y\subseteq M$ with $d_\h(Y, M)\leq\varepsilon$, we have
\[d_\gh(X,Y) \ge \min\left\{\tfrac{1}{2}d_\h(X,Y)-\tfrac{3}{2}\varepsilon, \tfrac{1}{6}\rho(M)-\tfrac{2}{3}\varepsilon\right\}.
\]
\end{corollary}
\begin{proof}
Applying the triangle inequality on the right side of Theorem~\ref{thm:main1}(b) we get
\begin{align*}
d_\gh(X,Y) &\ge \min\left\{\tfrac{1}{2}d_\h(X,M)-d_\h(Y,M), \tfrac{1}{6}\rho(M)-\tfrac{2}{3}d_\h(Y,M)\right\} \\
&\ge \min\left\{\tfrac{1}{2}[d_\h(X,Y)-d_\h(Y, M)]-d_\h(Y,M), \tfrac{1}{6}\rho(M)-\tfrac{2}{3}d_\h(Y,M)\right\} \\
&\ge \min\left\{\tfrac{1}{2}d_\h(X,Y)-\tfrac{3}{2}d_\h(Y,M), \tfrac{1}{6}\rho(M)-\tfrac{2}{3}d_\h(Y,M)\right\} \\
&\ge\min\left\{\tfrac{1}{2}d_\h(X,Y)-\tfrac{3}{2}\varepsilon, \tfrac{1}{6}\rho(M)-\tfrac{2}{3}\varepsilon\right\},\text{ since }d_\h(Y, M)\leq\varepsilon.
\end{align*}
\end{proof}

\section{Improving the convexity radius term via the filling radius}
\label{sec:fill-rad}

In~\cite{gromov1983filling}, Gromov defines the \emph{filling radius} of a manifold as follows.
Any metric space $(Z,d_Z)$ can be isometrically embedded into the space $L^\infty(Z)$ of all functions on $Z$, equipped with the supremum metric; this is called the Kuratowski embedding $\iota \colon Z \hookrightarrow L^\infty(Z)$ defined via $z\mapsto d_Z(z,\cdot)$.
The \emph{filling radius} of an $n$-dimensional manifold $M$ is the infimal scale parameter $r>0$ such that the inclusion $M \hookrightarrow B_{L^{\infty}(X)}(\iota(M);r)$ induces a non-injective homomorphism on $n$-dimensional homology, i.e., kills the fundamental class of $M$.
Here $B_{L^{\infty}(X)}(M;r)$ is the union of all balls in $L^\infty(X)$ of radius $r$ that are centered at a point in $\iota(M)$.
In~\cite{lim2020vietoris}, Lim, M\'{e}moli, and Okutan prove that the filling radius, denoted $\fr(M)$, can equivalently be defined as the supremal scale parameter $r>0$ such that the inclusion $\vr{M}{\varepsilon} \hookrightarrow \vr{M}{2r}$ induces an isomorphism on $n$-dimensional homology for all sufficiently small $\varepsilon>0$.
See~\cite[Proposition~9.6]{lim2020vietoris} for an application to Gromov--Hausdorff distances.

\begin{figure}[htb]
\centering
\includegraphics[width=3.5in]{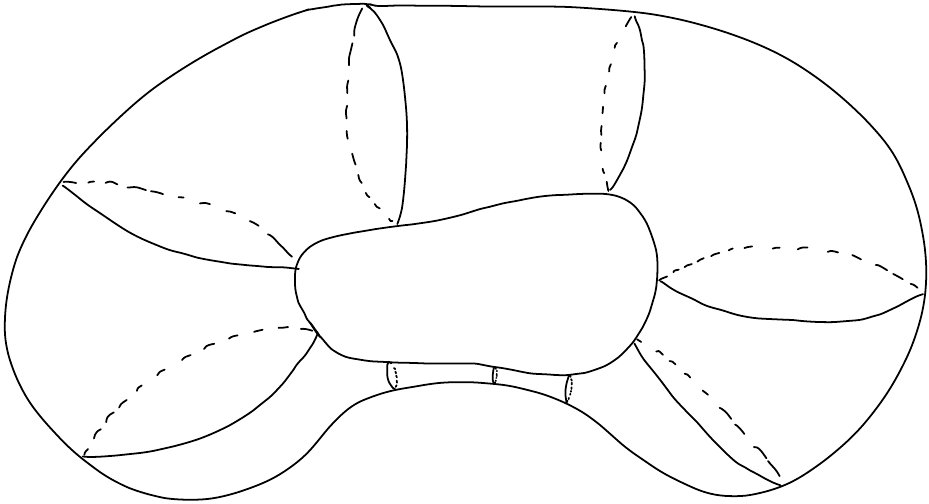}
\caption{
A two-dimensional torus $M$ whose convexity radius is small (determined by the thinnest part of the neck) and whose filling radius is large.
}
\label{fig:filling-radius}
\end{figure}

Theorem~\ref{thm:main3} replaces the $\frac{1}{6}\rho(M)$ convexity radius terms from Theorem~\ref{thm:main1} with a larger $\frac{1}{2}\rho(M)$ term.
The price we pay is introducing a new $\frac{1}{3}\fr(M)$ filling radius term.
In general, the filling radius $\fr(M)$ can be much larger than the convexity radius $\rho(M)$; see Figure~\ref{fig:filling-radius}.
That is why Theorem~\ref{thm:main3} can provide a large improvement over Theorem~\ref{thm:main1}.

\vspace{3mm}
\mainThree*

\begin{proof}
Part (a) is obtained from part (b) by letting $Y=M$.
Hence we prove (b).
We follow the same notation and almost the same argument as the proof of Theorem~\ref{thm:main1}(b).
As before, assume that $d_\gh(X,Y)$ is strictly smaller than the last two terms in the minimum.
Let $2\cdot d_\h(Y,M)<\delta<\min\{\rho(M)-2\cdot d_\gh(X,Y), \fr(M)-3\cdot d_\gh(X,Y)\}$.
Choose $r$ such that
\[2\cdot d_\gh(X,Y)<r<\min\left\{\rho(M)-\delta,\tfrac{2}{3}\fr(M)-\tfrac{2}{3}\delta\right\},\]
and choose $\varepsilon>0$ such that $r+\delta+\varepsilon<\rho(M)$ and $3r+2\delta+2\varepsilon<2\cdot\fr(M)$.
We get the same diagram: 
\begin{equation*}
\begin{tikzpicture} [baseline=(current  bounding  box.center)]
\centering
\node (k1) at (-4.8,-1.5) {$\vr{M}{\varepsilon}$};
\node (k2) at (-4.8,0) {$\vr{Y}{\delta+\varepsilon}$};
\node (k3) at (-1.85,0) {$\vr{X}{r+\delta+\varepsilon}$};
\node (k4) at (0,-1.5) {$\cech{X}{r+\delta+\varepsilon}$};
\node (k5) at (1.85,0) {$\vr{X}{2r+2\delta+2\varepsilon}$};
\node (k6) at (5.6,0) {$\vr{Y}{3r+2\delta+2\varepsilon}$};
\node (k7) at (5.6,-1.5) {$\vr{M}{3r+2\delta+2\varepsilon}$};
\draw[map] (k1) to node[auto] {$\of$} (k2);
\draw[map] (k2) to node[auto] {$\oh$} (k3);
\draw[rinclusion] (k3) to node[auto] {$\iota$} (k4);
\draw[rinclusion] (k4) to node[auto] {$\iota'$} (k5);
\draw[map] (k5) to node[auto] {$\og$} (k6);
\draw[rinclusion] (k6) to node[auto] {$\iota''$} (k7);
\draw[rinclusion, dashed] (k3) to (k5);
\end{tikzpicture}        
\end{equation*}
The composition is still contiguous to the inclusion $\vr{M}{\varepsilon}\hookrightarrow\vr{M}{3r+2\delta+2\varepsilon}$, which by the bound on the filling radius and~\cite{lim2020vietoris} induces an isomorphism on the fundamental class, i.e.\ on the $n$-dimensional homology group.
Furthermore, note that $r+\delta+\varepsilon<\rho(M)$ implies that the \v{C}ech complex $\cech{X}{r+\delta+\varepsilon}$ is homotopy equivalent to the union $B(X,r+\delta+\varepsilon)$.
Arguing as before, we obtain $d_\gh(X,Y)\ge\frac{1}{2}d_\h(X,M)-d_\h(Y,M)$, as desired.
\end{proof}

\section{Improving the Hausdorff distance constant via Jung's theorem}
\label{sec:Jung}

We review Jung's theorem, which will allow us to improve the constant on the term $\frac{1}{2}d_\h(X,M)$ in Theorem~\ref{thm:main1} to a better constant in Theorem~\ref{thm:main4}.
Jung's theorem allows us to produce a tighter interleaving between \v{C}ech and Vietoris--Rips complexes.

\subsection*{Jung's theorem}
The definition of sectional curvatures of an abstract manifold $M$ uses a lot of machinery from Riemannian geometry.
We skip the definition here, referring the reader to a textbook on the subject, e.g.,~\cite[Chapter 9]{BishopRichardL2001Gom}. 
For a point $p\in M$ and unit norm vectors $u,v\in T_p(M)$ in the tangent space of $M$ at $p$, the sectional curvature at $p$ along the plane spanned by $u$ and $v$ is denoted by $\kappa_p(u,v)$.
Intuitively, it measures the Gaussian curvature at $p$ for the $2$-dimensional submanifold with tangent plane spanned by $u$ and~$v$.
For example, the $n$-sphere of radius $R$ has a constant sectional curvature of $1/R^2$.

Let $\kappa(M)\in\R$ denote the supremum of the set of sectional curvatures $\kappa_p(u,v)$ across all $p\in M$ and all $u,v \in T_p(M)$.
Since $M$ is compact, it can be shown that $\kappa(M)$ is bounded; see for example~\cite[p.~166]{BishopRichardL2001Gom}. 

\subsection*{Jung's theorem for manifolds}
For a compact subset $A\subseteq M$, the diameter satisfies $\diam(A)<\infty$.
We define its \emph{circumradius} in $M$ to be
\[
R(A)\coloneqq \inf_{m\in M}\max_{a\in A}d_M(a,m).
\]
Intuitively, the circumradius is the radius of a smallest closed metric ball (when one exists) in $M$ that contains $A$. 

A point $c\in M$ satisfying $\max_{a\in A}d_M(a ,c)=R(A)$ is called a \emph{circumcenter} of $A$, and is denoted by $c(A)$.
For $A$ compact, the circumradius is uniquely defined, but a circumcenter may not exist.
When $M=\R^n$, however, the circumcenter exists uniquely. 
Moreover, for a compact Euclidean subset $A\subset\R^d$, the classical Jung's theorem~\cite[Theorem 2.6]{danzer1963helly} states that $R(A)\leq\sqrt{\tfrac{n}{2(n+1)}}~\diam(A)$.

Jung's theorem was further extended by Dekster in
\cite{Dekster1985AnEO,Dekster1995TheJT,Dekster1997}, first for compact subsets of Riemannian manifolds with constant sectional curvatures, and then for Alexandrov spaces of curvature bounded above.
A corollary in~\cite[Section 2]{Dekster1997} affirms that the circumcenter $c(A)$ exists (possibly non-uniquely) for a compact set $A$ if:
\begin{enumerate}
\item $A$ is contained in the interior of a compact geodesically convex domain
$C\subseteq M$, and
\item $\diam(C)<\frac{2\pi}{3\sqrt{\kappa(M)}}$ if $\kappa(M)>0$.
\end{enumerate}
Moreover, $c(A)$ belongs to the interior of $C$, and we have the
following bound on the circumradius of $A$ in terms of the diameter of $A$.

\vspace{3mm}
\begin{theorem}[Extended Jung's theorem~\cite{Dekster1997}]
\label{thm:Jung}
Let $M$ be a compact, connected, $n$--dimensional manifold with sectional curvatures bounded above by $\kappa\in\R$. 
Let $A\subseteq M$ be compact with $\diam(A)<\rho(M)$ if $\kappa\le 0$ and $\diam(A)<\min\left\{\rho(M),\tfrac{2\pi}{3\sqrt{\kappa}}\right\}$ if $\kappa>0$.
Then the circumcenter $c(A)$ exists in $M$, and
\[
\diam(A)\geq\begin{cases}
\frac{2}{\sqrt{-\kappa}}\sinh^{-1}\left(\sqrt{\frac{n+1}{2n}}
\sinh\left(\sqrt{-\kappa}\;R(A)\right)\right) &\text{for }\kappa<0\\
2R(A)\sqrt{\frac{n+1}{2n}} &\text{for }\kappa=0 \\
\frac{2}{\sqrt{\kappa}}\sin^{-1}\left(\sqrt{\frac{n+1}{2n}}
\sin\left(\sqrt{\kappa}\;R(A)\right)\right) 
&\text{for }\kappa>0\text{ and }R(A)\in
\left[0,\tfrac{\pi}{2\sqrt{\kappa}}\right].
\end{cases}
\]
\end{theorem}

This theorem (and the proof of Lemma~\ref{lem:circum} in Appendix~\ref{app:additional}) leads us to define the following constants.

\vspace{3mm}
\begin{definition}
\label{def:alpha}
For each integer $n\ge 1$ and $\kappa\in\R$, we define
\[
\alpha(n,\kappa)=\begin{cases}
\sqrt{\tfrac{n+1}{2n}}&\text{for }\kappa\leq0 \\
\sqrt{\tfrac{n+1}{2n}}\sin\left(\tfrac{\pi}{2}\sqrt{\frac{\kappa}{\kappa+1}}\right)/\left(\tfrac{\pi}{2}\sqrt{\frac{\kappa}{\kappa+1}}\right)&\text{for }\kappa>0.
\end{cases}
\]
\end{definition}
\noindent Note that $\alpha(n,\kappa)$ is continuous as a function of $\kappa$, since $\alpha(n,\kappa)\to\sqrt{\frac{n+1}{2n}}$ as $\kappa\to0$. 

Note that $\sqrt{\tfrac{n+1}{2n}}$ is a decreasing function of $n$ that approaches its infimum value $\frac{1}{\sqrt{2}}$ as $n\to \infty$.
So for $\kappa\le 0$, we have $\frac{1}{\sqrt{2}}\le \alpha(n,\kappa)\le 1$.
For $\kappa> 0$, we note that the function $\sin\left(\tfrac{\pi}{2}\sqrt{\frac{\kappa}{\kappa+1}}\right)/\left(\tfrac{\pi}{2}\sqrt{\frac{\kappa}{\kappa+1}}\right)$ is a decreasing function of $\kappa$ that approaches its infimumum value $\frac{2}{\pi}$ as $\kappa\to\infty$.
Therefore, for $n,\kappa$ arbitrary, we have $0.45 \approx\frac{\sqrt{2}}{\pi}= \frac{2}{\pi}\cdot\frac{1}{\sqrt{2}} \le \alpha(n,\kappa)\le 1$.
Our Theorem~\ref{thm:main4} only improves upon the $\frac{1}{2}\rho(M)$ term in Theorem~\ref{thm:main1} when $\alpha(n,\kappa) \ge \frac{1}{2}$, but this happens for a large range of $n$ and $\kappa$ values, and in particular whenever $\kappa\le 0$.

\vspace{3mm}
\begin{lemma}[Circumradius]
\label{lem:circum}
Let $A\subseteq M$ be compact with $\diam(A)< \tau$, where $\tau=\rho(M)$ for $\kappa\le 0$ and $\tau=\min\left\{\rho(M),\tfrac{\pi}{2\sqrt{\kappa+1}}\right\}$ for $\kappa>0$.
Then the circumcenter $c(A)$ exists in $M$, and the diameter and circumradius of $A$ satisfy
$\diam(A)\geq2\alpha(n,\kappa)\cdot R(A)$.
\end{lemma}
\vspace{3mm}

See Appendix~\ref{app:additional} for the proof.
The following lemma, illustrated in Figure~\ref{fig:rips-cech-interleaving}, generalizes Lemma \ref{lem:circle-cech} from the circle case to general manifolds.

\begin{figure}[htb]
\centering
\includegraphics[width=5in]{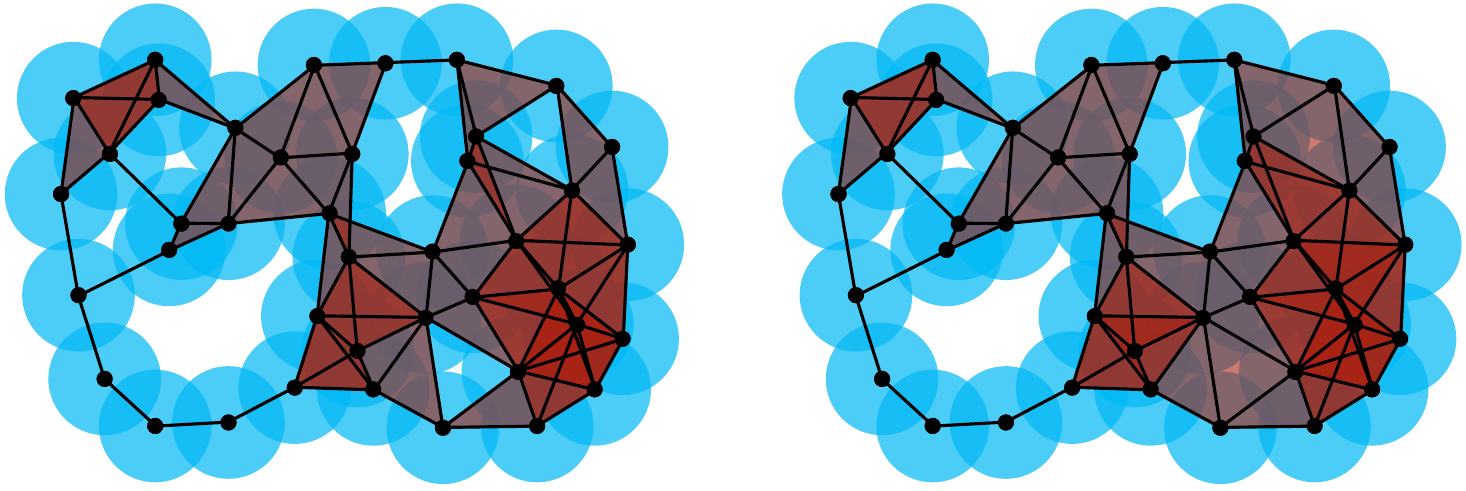}
\caption{An inclusion of simplicial complexes 
$\cech{X}{\frac{r}{2}}\hookrightarrow\vr{X}{r}$.
We also have an inclusion $\vr{X}{r}\hookrightarrow\cech{X}{\tfrac{r}{2\alpha(n,\kappa)}}$ obtained by increasing the size of the balls.
}
\label{fig:rips-cech-interleaving}
\end{figure}

\vspace{3mm}
\begin{lemma}
\label{lem:rips-cech}
Let $M$ be a closed Riemannian manifold with $X\subseteq M$. 
For any $0<r<\tau$, we have the inclusion $\vr{X}{r} \hookrightarrow \cech{X}{\tfrac{r}{2\alpha(n,\kappa)}}$, where $\tau=\rho(M)$ if $\kappa \le 0$ and
$\tau=\min\left\{\rho(M),\tfrac{\pi}{2\sqrt{\kappa+1}}\right\}$ if $\kappa>0$.
\end{lemma}

\begin{proof}
Let $\sigma\in\vr{X}{r}$, so $\diam(\sigma) < r <\tau$.
From Lemma~\ref{lem:circum}, it follows that $c(\sigma)$ exists in $M$, moreover $R(\sigma)\le\frac{\diam(\sigma)}{2\alpha(n,\kappa)}<\frac{r}{2\alpha(n,\kappa)}$.
Consequently, $\sigma$ must be contained in the metric ball of radius $\frac{r}{2\alpha(n,\kappa)}$ centered at $c(\sigma)$.
Therefore, $\sigma\in\cech{X}{\tfrac{r}{2\alpha(n,\kappa)}}$.
\end{proof}

\subsection*{Manifolds with an improved constant}
\label{sec:manifold}

The following theorem improves upon the constant $\frac{1}{2}$ in the Hausdorff distance term $\tfrac{1}{2}d_\h(X,M)$ in Theorem~\ref{thm:main1}.

\vspace{3mm}
\mainFour*
\vspace{3mm}

We note that $\alpha(n,\kappa)=\alpha(1,0)=1$ when $M=S^1$ is the unit circle.

\begin{proof}
Let $\alpha=\alpha(n,\kappa)$.
For all $2\cdot d_\h(Y,M)<\delta<\alpha\tau-(2\alpha+2)d_\gh(X,Y)$, we will show if $d_\gh(X,Y)<\frac{\alpha\tau-\delta}{2\alpha+2}$, then $d_\gh(X,Y)\geq \alpha \cdot d_\h(X,M)-\frac{\delta}{2}$.

Fix $r$ such that $2\cdot d_\gh(X,Y) < r < \frac{\alpha\tau-\delta}{\alpha+1}$, and fix $\varepsilon>0$ so that 
\begin{equation}
\label{eq:zeta-1}
\frac{(\alpha+1)r+\delta+\varepsilon}{\alpha} < \tau.
\end{equation}
We produce the following sequence of continuous maps.
Since $r+\delta+\varepsilon<(\alpha+1)r+\delta+\varepsilon<\alpha\tau\le \tau$ (where the middle step is from~\eqref{eq:zeta-1}), Lemma~\ref{lem:rips-cech} gives the inclusion $\iota$.
After doubling the scale, the inclusion $\iota'$ always holds for \v{C}ech and Vietoris--Rips complexes.
We note that $\iota'\circ\iota$ is equal to the dashed inclusion, so the middle triangle commutes.
Since $2\cdot d_\gh(X,Y) < r$, we use Lemma~\ref{lem:correspondence-vr} to get the maps $\oh$ and $\og$, which satisfy that $\og\circ\iota'\circ\iota\circ\oh$ is homotopic to the inclusion map.
\begin{equation*}
\begin{tikzpicture} [baseline=(current  bounding  box.center)]
\centering
\node (k1) at (-4.7,-1.5) {$\vr{M}{\varepsilon}$};
\node (k2) at (-4.7,0) {$\vr{Y}{\delta+\varepsilon}$};
\node (k3) at (-1.6,0) {$\vr{X}{r+\delta+\varepsilon}$};
\node (k4) at (0,-1.5) {$\cech{X}{\tfrac{r+\delta+\varepsilon}{2\alpha}}$};
\node (k5) at (1.6,0) {$\vr{X}{\tfrac{r+\delta+\varepsilon}{\alpha}}$};
\node (k6) at (5,0) {$\vr{Y}{r+\tfrac{r+\delta+\varepsilon}{\alpha}}$};
\node (k7) at (5,-1.5) {$\vr{M}{\tfrac{(\alpha+1)r+\delta+\varepsilon}{\alpha}}$};
\draw[map] (k1) to node[auto] {$\of$} (k2);
\draw[map] (k2) to node[auto] {$\oh$} (k3);
\draw[rinclusion] (k3) to node[auto] {$\iota$} (k4);
\draw[rinclusion] (k4) to node[auto] {$\iota'$} (k5);
\draw[map] (k5) to node[auto] {$\og$} (k6);
\draw[rinclusion] (k6) to node[auto] {$\iota''$} (k7);
\draw[rinclusion, dashed] (k3) to (k5);
\end{tikzpicture}        
\end{equation*}
Finally, Lemma~\ref{lem:rips-hausdorff} gives the maps $\of$ and $\iota''$, which satisfy that $\iota''\circ\og\circ\iota'\circ\iota\circ \oh\circ\of$ is homotopic to the inclusion $\vr{M}{\varepsilon}\hookrightarrow\vr{M}{\tfrac{(\alpha+1)r+\delta+\varepsilon}{\alpha}}$. 
By Hausmann's theorem (Theorem~\ref{thm:hausmann}) and the fact that $0<\varepsilon<\tfrac{(\alpha+1)r+\delta+\varepsilon}{\alpha}<\tau\leq \rho(M)$ from~\eqref{eq:zeta-1}, this inclusion is a homotopy equivalence between spaces homotopy equivalent to $M$.
Hence the composition $\iota''\circ\og\circ\iota'\circ\iota\circ \oh\circ\of$ preserves the fundamental class in $n$-dimensional homology.

If we had $\frac{r+\delta+\varepsilon}{2\alpha} < d_\h(X,M)$, then the collection $\left\{B\left(x,\frac{r+\delta+\varepsilon}{2\alpha}\right)\right\}_{x\in X}$ of open balls would not cover $M$.
By the nerve lemma, $\cech{X}{\frac{r+\delta+\varepsilon}{2\alpha}}$ would be homotopy equivalent to an open proper subset of $M$, implying $H_n\left(\cech{X}{\frac{r+\delta+\varepsilon}{2\alpha}}\right)=0$ by Lemma~\ref{lem:proper-subset}, contradicting the fact that the fundamental class is preserved.
Therefore, it must be that $\frac{r+\delta+\varepsilon}{2\alpha} \ge d_\h(X,M)$ for all $2\cdot d_\gh(X,Y) < r < \frac{\alpha\tau-\delta}{\alpha+1}$ and sufficiently small $\varepsilon>0$.
So, $2 \cdot d_\gh(X,Y)+\delta \ge 2\alpha\cdot d_\h(X,M)$.
Thus, $d_\gh(X,Y) \ge \alpha \cdot d_\h(X,M)-d_\h(Y,M)$.
\end{proof}

\begin{example}
\label{ex:tori}
The $n$-dimensional flat torus $T^n= S^1\times\ldots\times S^1$ is endowed with a Riemannian metric that makes the sectional curvatures at each point zero, giving $\alpha(n,\kappa)=\sqrt{\tfrac{n+1}{2n}}$.
The convexity radius is $\rho(T^n)=\frac{\pi}{2}$.
Theorem~\ref{thm:main4} gives that for any $X,Y\subseteq T^n$,
\[d_\gh(X,Y) \geq \min\left\{\sqrt{\tfrac{n+1}{2n}}d_\h(X,M)-d_\h(Y,M),\frac{\sqrt{\tfrac{n+1}{2n}}\tfrac{\pi}{2}-2d_\h(Y,M)}{2\sqrt{\tfrac{n+1}{2n}}+2}\right\}.\]
\end{example}

\section{Ratio can be arbitrarily small in general}
\label{sec:ratio}

We next show that in general, the Hausdorff and Gromov--Hausdorff distances between general compact metric spaces and a subspace thereof can be arbitrarily far apart.
Indeed, in Theorem~\ref{thm:main5} we show that for any $\varepsilon>0$, there exists a compact metric space $Z$ and a subset $X\subseteq Z$ with $d_\gh(X,Z) < \varepsilon \cdot d_\h(X,Z)$.

The ``compact'' hypothesis makes this result meaningful, since compact metric spaces do not admit isometries to proper subsets.
By contrast, note that the non-compact metric space $[0,\infty)$ with the Euclidean metric has an isometric subspace $[n,\infty)$ satisfying $d_\gh([0,\infty),[n,\infty))=0$ and $d_\h([0,\infty),[n,\infty))=n$ for all $n>0$.

\vspace{3mm}
\mainFive*
\vspace{3mm}

\begin{proof}
We consider $\R^n$ with the Euclidean $\ell_2$-metric.
Denote the standard basis vectors in $\R^n$ by $e_1, \dots, e_n$, and let $x_j = \sum_{i=1}^j i\cdot e_i$ \ for $j\in\{1,\ldots ,n\}$.
Now let $Z = \{x_1, \dots, x_n\}$ and $X = \{x_1,\dots, x_{n-1}\} \subseteq Z$.
For the Hausdorff distance between $X$ and $Z$, we have that $d_\h(X,Z) = \|x_n - x_{n-1}\| = n$.
Consider the linear map $f \colon \R^n \to \R^n$ defined by $f(e_i) = e_{i+1}$ for $i \in \{1,\dots, n-1\}$ and $f(e_n) = e_1$.
For $j \in \{2,\dots,n\}$, we compute $\|f(x_{j-1}) - x_j\| = \|\sum_{i=1}^j e_i\| = \sqrt{j}$.
Further, $\|f(x_1) - x_1\| = \|e_2 - e_1\| = \sqrt{2}$.
Also $\min\{\|f(x_i) - x_n\|~:~1\le i\le n-1\} = \sqrt{n}$.
Thus we get that $d_\h(f(X),Z) = \sqrt{n}$.
Since $f$ is an isometry of $\R^n$ this proves that ${d_\gh(X,Z) \le \sqrt{n}}$.
This proves the theorem, by picking $n$ large enough so that $\frac{1}{\sqrt{n}}<\varepsilon$.
\end{proof}

\begin{remark}
\label{rem:boundary-and-not-dense}
The proof above shows that the ratio between the Hausdorff distance of compact subsets $X\subseteq Z$ in $\R^n$ and the infimal Hausdorff distance of isometric copies of these subsets in~$\R^n$ may be arbitrarily far from one.
The sets $X$ and $Z$ in the proof above can be replaced by the compact simplices $\conv \ X$ and $\conv \ Z$ with the same Hausdorff distances, that is, $d_\h(\conv \ X, \conv \ Z) = n$, while $d_\h(f(\conv \ X), \conv \ Z) = \sqrt{n}$.
We note that $\conv \ X$ and $\conv \ Z$ are manifolds with boundary.
\end{remark}

\section{Conclusion and open questions}
\label{sec:conclusion}

We have furthered the study of the relationship between the Hausdorff and Gromov--Hausdorff distances, particularly for sufficiently dense subsets of a closed Riemannian manifold.
We conclude with a discussion of how this project arose, along with a list of open questions.

In our paper, we have used the convenient definition of the Gromov--Hausdorff distance as $d_{\gh}(X,Y) = \tfrac{1}{2} \inf_{C \subseteq X \times Y}\dis(C)$,
where the infimum is taken over all correspondences between $X$ and $Y$.
It was observed in~\cite{kalton1999distances} that the Gromov--Hausdorff distance can be equivalently defined as
\[d_{\gh}(X,Y) = \tfrac{1}{2}\inf_{g,h}\max\left\{\dis(g),\dis(h),\codis(g,h)\right\},\]
where $g\colon X \to Y$ and $h\colon Y \to X$ are possibly discontinuous functions.
Here the distortion of the function $g$ is
\[\dis(g)=\sup_{x,x'\in X}|d_{X}(x,x')-d_{Y}(g(x),g(x'))|\]
(and similarly for $\dis(h)$), and the codistortion of the pair of functions $g$ and $h$ is
\[\codis(g,h)=\sup_{x\in X,y\in Y}|d_{X}(x,h(y))-d_{Y}(g(x),y)|.\]
The papers~\cite{lim2021gromov,GH-BU-VR} on Gromov--Hausdorff distances between spheres gave lower bounds on $d_\gh(X,Y)$ by lower bounding $\dis(h)$, the distortion of functions $Y \to X$, in particular in a setting where $Y$ was a higher-dimensional sphere, $X$ was a lower-dimensional sphere, and it sufficed to consider odd functions.
Topological obstructions were used in both~\cite{lim2021gromov,GH-BU-VR} to bound $\dis(h)$, and~\cite{GH-BU-VR} obtained these topological obstructions by first converting the (possibly discontinuous) function $h$ into a continuous map between Vietoris--Rips complexes.
By contrast, our paper on Hausdorff versus Gromov--Hausdorff distances arose in part while looking for lower bounds on $d_\gh(X,Y)$ that also involved $\codis(g,h)$.
We succeeded by noting that control on $\codis(g,h)$ allowed us to show that a (possibly discontinuous) composition $g\circ h\colon Y\to Y$ produced a continuous map on Vietoris--Rips complexes $\overline{g}\circ\overline{h}\colon\vr{Y}{r}\to\vr{Y}{r'}$ with $r \le r'$ that was contiguous to the inclusion map $\vr{Y}{r}\hookrightarrow\vr{Y}{r'}$ (see Lemma~\ref{lem:correspondence-vr}), and similarly for \v{C}ech complexes (Lemma~\ref{lem:correspondence-cech}).

Our investigations of Hausdorff versus Gromov--Hausdorff distances have touched upon a number of naturally arising open questions which we believe merit further study.
We list these questions here:

\vspace{3mm}
\begin{question}
The constant $\alpha(n,\kappa)$ in Theorem~\ref{thm:main4} attains the value $1$ when $M$ is the circle, which cannot be improved.
Are the constants $\alpha(n,\kappa)$ optimal in some sense for more general manifolds $M$?
\end{question}

\vspace{3mm}
\begin{question}
In Theorem~\ref{thm:main2}, we have $d_\gh(X,S^1)=d_\h(X,S^1)$ for $X\subseteq S^1$ with $d_\h(X,S^1)<\frac{\pi}{6}$.
Is $\frac{\pi}{6}$ optimal in order to achieve equality, i.e., can we construct an example subset $X$ with density slightly below $\frac{\pi}{6}$ and $d_\gh(X,S^1) < d_\h(X,S^1)$?
More generally, what happens to the ratio
\[\inf\left\{\frac{d_\gh(X,S^1)}{d_\h(X,S^1)}~:~X\subseteq S^1\text{ with }d_\h(X,S^1)\ge \delta\right\}\]
as $\delta$ decreases beneath $\frac{\pi}{6}$?
\end{question}

\vspace{3mm}
\begin{question}
Our results are for a manifold $M$ equipped with the Riemannian metric.
If $M$ is instead a manifold embedded in Euclidean space, and equipped with the Euclidean metric, then can one prove analogous results?
\end{question}

\vspace{3mm}
\begin{question}
What if $M$ is not a manifold, but instead say a graph, or a stratified space, or a length space?
When can we prove that for dense enough samplings $X,Y\subseteq M$, the Gromov--Hausdorff distance $d_\gh(X,Y)$ is lower bounded in terms of the Hausdorff distance?
\end{question}

\vspace{3mm}
\begin{question}
\label{ques:not-dense}
Let $M$ be a connected, closed Riemannian manifold.
As explained at the end of Section~\ref{sec:manifold-simple-two}, 
Theorem~\ref{thm:main1}(b) implies that if $X$ and $Y$ are two isometric subsets, and if $Y$ dense enough ($d_\h(Y,M)<\frac{1}{4}\rho(M)$), then the Hausdorff distance from $X$ to $M$ and from $Y$ to $M$ cannot differ by more than a factor of two: $d_\h(X,M)\le 2\cdot d_\h(Y,M)$.
The example in Figure~\ref{fig:isometry-X-Y} shows that a density assumption like $d_\h(Y,M)<\frac{1}{4}\rho(M)$ is needed, since otherwise we could have $d_\h(X,M)\approx 4\cdot d_\h(Y,M)$.
For $X$ and $Y$ two isometric subsets of $M$ with no density assumption, how large can the ratio $\sup\left\{\tfrac{d_\h(X,M)}{d_\h(Y,M)}~|~0<d_\h(Y,M)\le d_\h(X,M)\right\}$ be?
\end{question}

\vspace{3mm}
\begin{question}
Are there variants of our results (Theorems~\ref{thm:main1},~\ref{thm:main3}, and~\ref{thm:main4}) for subsets of a complete manifold without boundary (that could be unbounded)?
For example, are there variants of our results for subsets of $\R^n$?
Does Borel--Moore homology (or alternatively cohomology with compact support), in which every oriented manifold has a fundamental class, play a role?
\end{question}

\vspace{3mm}
\begin{question}
As a specific example of the question above, are there versions of our results that hold for two periodic point sets in $\R^n$, where a periodic point set is defined as the image of a set of points in $[0,2\pi]^n$ under the natural action of $(2\pi\Z)^n$ on $\R^n$?
One can recast the periodic point set as a point set in the Torus $T^n$, and then proceed as in Example~\ref{ex:tori}.
But it is not clear whether non-periodic correspondences on the point sets on $\R^n$ can have smaller distortion than periodic correspondences induced from correspondences on point sets in $T^n$.

More generally, let $M$ be a complete manifold without boundary.
Suppose $G$ is a group that acts properly and by isometries on $M$, and that $X,Y\subseteq M$ are $G$-invariant subsets.
Then the quotient space $M/G$ is a metric space that contains $X/G$ and $Y/G$.
What can be said about the relationship between $d_\gh(X/G,Y/G)$ and $d_\gh(X,Y)$, especially for $X,Y\subseteq M$ sufficiently dense?
We note that if $M/G$ is furthermore a compact manifold without boundary, then Theorem~\ref{thm:main1} applies to give that
\begin{align*}
&d_\gh(X/G,Y/G)\\
\ge& \min\left\{\tfrac{1}{2}d_\h(X/G,M/G)-d_\h(Y/G,X/G), \tfrac{1}{6}\rho(M/G)-\tfrac{2}{3}d_\h(Y/G,M/G)\right\}\\
=& \min\left\{\tfrac{1}{2}d_\h(X,M)-d_\h(Y,M), \tfrac{1}{6}\rho(M/G)-\tfrac{2}{3}d_\h(Y,M)\right\}.
\end{align*}
\end{question}

\vspace{3mm}
\begin{question}
What if $M$ is not a closed manifold, but instead a compact manifold with boundary $\partial M \neq \emptyset$?

Let $X,Y \subseteq M$ with $\partial M \neq \emptyset$.
One thing that could be done in this situation is to define $N$ to be the closed manifold ($\partial N = \emptyset$) obtained by gluing two copies of $M$ together along their common boundary.
Now, define $X_N,Y_N\subseteq N$ to be the subsets obtained by duplicating $X$ and $Y$ (although points in the boundary of $M$ will not be duplicated).
Suppose furthermore that $N$ is a closed Riemannian manifold with convexity radius $\rho(N)$.
Then Theorem~\ref{thm:main1} applies to give that
\begin{align*}
d_\gh(X_N,Y_N)
&\ge \min\left\{\tfrac{1}{2}d_\h(X_N,N)-d_\h(Y_N,N), \tfrac{1}{6}\rho(N)-\tfrac{2}{3}d_\h(Y_N,N)\right\}\\
&= \min\left\{\tfrac{1}{2}d_\h(X,M)-d_\h(Y,M), \tfrac{1}{6}\rho(N)-\tfrac{2}{3}d_\h(Y,M)\right\}.
\end{align*}
What can be said about the relationship between $d_\gh(X,Y)$ and $d_\gh(X_N,Y_N)$, especially for $X,Y\subseteq M$ sufficiently dense?
\end{question}

\vspace{3mm}
\begin{question}
For $X,Y\subseteq \R^n$, there is a nesting $d_\gh(X,Y) \le d_{\mathrm{H,iso}}^{\R^n}(X,Y) \le d_\h(X,Y)$, where $d_{\mathrm{H,iso}}^{\R^n}$ allows one to align $X$ and $Y$ only via Euclidean isometries; see~\cite{majhi2023GH,memoli2008euclidean}.
Similarly, if $M$ is a manifold with symmetries, and $X,Y\subseteq M$, there is a nesting $d_\gh(X,Y) \le d_{\mathrm{H,iso}}^M(X,Y) \le d_\h(X,Y)$, where $d_{\mathrm{H,iso}}^M$ allows one to align $X$ and $Y$ only via isometries of the manifold $M$.
Can the techniques in our paper be extended to place better bounds on $d_{\mathrm{H,iso}}^M$?
\end{question}

\section*{Acknowledgements}

HA and SM would like to thank the Department of Mathematics at the University of Florida for hosting a research visit.
HA was funded in part by the Simons Foundation's Travel Support for Mathematicians program.
FF was funded in part by NSF CAREER Grant DMS 2042428.

\bibliographystyle{plain}
\bibliography{HausdorffVsGromovHausdorff.bib}

\appendix

\section{Proofs of interleaving lemmas}
\label{app:proofs-interleavings}

This appendix contains the proofs of Lemmas~\ref{lem:correspondence-vr} and~\ref{lem:rips-hausdorff}.
We include these proofs here so that our paper is self-contained, but we also refer the reader to~\cite{ChazalDeSilvaOudot2014} for more details.

\begin{proof}[Proof of Lemma~\ref{lem:correspondence-vr}]
Let $X$ and $Y$ be metric spaces and let $r > 2\cdot d_\gh(X,Y)$.
We must show that for any $\varepsilon > 0$ and $\nu\geq0$, there exist simplicial maps
\[\vr{Y}{\varepsilon} \xlongrightarrow{\oh} \vr{X}{r+\varepsilon}
\xhookrightarrow{\iota_\nu}
\vr{X}{r+\varepsilon+\nu}
\xlongrightarrow{\og} \vr{Y}{2r+\varepsilon+\nu}\] 
such that the composition $\oh\circ\iota_\nu\circ\og$ is contiguous to the inclusion $\iota\colon\vr{Y}{\varepsilon}\hookrightarrow\vr{Y}{2r+\varepsilon+\nu}$.

Since $2\cdot d_\gh(X,Y)<r$, there exists a correspondence $C\subseteq X\times Y$ with $\dis(C)<r$.
We can define (possibly discontinuous) functions $g\colon X\to Y$ and $h\colon Y\to X$ such that $(x,g(x))\in C$ for any $x\in X$ and $(h(y),y)\in C$ for any $y\in Y$.

We first show that the map $h$ extends to a simplicial map
$\oh\colon \vr{Y}{\varepsilon}\to \vr{X}{r+\varepsilon}$.
Take a simplex $\sigma=[y_0,\ldots,y_k]\in\vr{Y}{\varepsilon}$. 
By the definition of Vietoris--Rips complex, the diameter of $\sigma$ is smaller than than $\varepsilon$, i.e., $d_Y(y_i,y_j)<\varepsilon$ for $0 \le i,j \le k$.
Since $(h(y_i),y_i),(h(y_j),y_j)\in C$, by the definition of distortion we have
\[
d_X(h(x_i),h(x_j))\leq d_Y(y_i,y_j)+r < r+\varepsilon.
\]
So, the diameter of $\oh(\sigma)=[h(y_0),\ldots,h(y_k)]$ is smaller than $r+\varepsilon$, implying that $\oh(\sigma)\in\vr{X}{r+\varepsilon}$, and so $\oh$ is a simplicial map.

Since $\nu\geq0$, the existence of the natural inclusion map $\iota_\nu$ in the middle is evident.
The proof that $g$ extends to a simplicial map $\og$ is analogous to the $\oh$ case.

To show that the simplicial maps $\og\circ\iota_\nu\circ\oh$ and $\iota$ are contiguous, we take an arbitrary $k$-simplex $\sigma=[y_0,\ldots,y_k]\in\vr{Y}{\varepsilon}$, meaning that $d_Y(y_i,y_j) < \varepsilon$ for all $i,j$.
Note that $\iota(\sigma)\in \vr{Y}{2r+\varepsilon+\nu}$ and $(\og\circ\iota_\nu\circ\oh)(\sigma) \in\vr{Y}{2r+\varepsilon+\nu}$.
For any $i,j$, we also have
\begin{align*}
d_Y(y_i,(g\circ h)(y_j))
&\leq d_X(h(y_i),h(y_j))+r && \text{ since }(h(y_i),y_i),(h(y_j),g(h(y_j)))\in C \\
&\leq d_Y(y_i,y_j)+ 2r && \text{ since }(h(y_i),y_i),(h(y_j),y_j)\in C  \\
& < 2r+\varepsilon \leq 2r+\varepsilon+\nu.
\end{align*}
As a result, the simplex 
\[
    \iota(\sigma)\cup(\og\circ\iota_\nu\circ\oh)(\sigma)=[y_0,\ldots,y_k,(g\circ h)(y_0),\ldots,(g\circ h)(y_k)]
\]
is a simplex in $\vr{Y}{2r+\varepsilon+\nu}$.
Therefore, the maps are contiguous.
\end{proof}

\begin{proof}[Proof of Lemma~\ref{lem:rips-hausdorff}]
Let $Z$ be a metric space, and let $Y\subseteq Z$ satisfy $r > 2\cdot d_\h(Y,Z)$.
We must show that for any $\varepsilon > 0$ and $\nu \ge 0$, there exist simplicial maps
\[\vr{Z}{\varepsilon} \xlongrightarrow{\of} \vr{Y}{r+\varepsilon} \xhookrightarrow{\iota_\nu} \vr{Y}{r+\varepsilon+\nu}
\xhookrightarrow{\iota} \vr{Z}{r+\varepsilon+\nu}\]
such that the composition $\iota \circ \iota_\nu \circ \of$ is contiguous to the inclusion $\iota'\colon \vr{Z}{\varepsilon}\hookrightarrow\vr{Z}{r+\varepsilon+\nu}$.

Since $d_\h(Y,Z)<\frac{r}{2}$, we can define a (possibly discontinuous) function $f\colon Z\to Y$ such that $d_Z(z,f(z))<\frac{r}{2}$ for any $z\in Z$.

We first show that the map $f$ extends to a simplicial map
$\of\colon \vr{Z}{\varepsilon}\to \vr{Y}{r+\varepsilon}$.
For $\sigma=[z_0,\ldots,z_k]\in\vr{Z}{\varepsilon}$, we have
\[
d_Z(f(z_i), f(z_j))\leq d_Z(f(z_i),z_i)+d_Z(z_i,z_j)+d_Z(z_j,f(z_j)) < \tfrac{r}{2}+\varepsilon+\tfrac{r}{2}=r+\varepsilon.
\]
So, the diameter of $\of(\sigma)=[f(z_0),f(z_1),\ldots,f(z_k)]$ is smaller than $r+\varepsilon$, giving $\of(\sigma)\in\vr{Y}{r+\varepsilon}$.
Therefore, $\of$ is a simplicial map.

In order to show that the simplicial maps $(\iota \circ \iota_\nu \circ \of)$ and $\iota'\colon \vr{Z}{\varepsilon}\hookrightarrow\vr{Z}{r+\varepsilon+\nu}$ are contiguous, we take an arbitrary simplex $\sigma=[z_0,\ldots,z_k]\in\vr{Z}{\varepsilon}$, i.e., $d_Z(z_i,z_j) < \varepsilon$ for all $i,j$.
Since $\of$ is a simplicial map, we have $d_Z(f(z_i),f(z_j)) < r+\varepsilon$.
We also note that for any $i,j$,
\[
d_Z(z_i,f(z_j))
\leq d_Z(z_i,z_j)+d_Z(z_j,f(z_j)) 
< \varepsilon + \tfrac{r}{2} < \varepsilon+r.
\]
As a result, the simplex 
$\iota'(\sigma)\cup (\iota\circ\iota_\nu\circ\of)(\sigma)=[z_0,z_1,\ldots,z_k,f(z_0),f(z_1),\ldots,f(z_k)] $
is in $\vr{Z}{r+\varepsilon+\nu}$.
Therefore, the maps are contiguous.
\end{proof}

\section{Additional proofs}
\label{app:additional}

As described at the end of Section~\ref{sec:circle}, we give an additional proof of Theorem~\ref{thm:main2}(a) using winding fractions.

\begin{proof}[Proof of Theorem~\ref{thm:main2}(a) using winding fractions]

We will show if $d_\gh(X,S^1)<\frac{\pi}{6}$, then $d_\gh(X,S^1)\geq d_\h(X,S^1)$.

Let $2\cdot d_\gh(X,S^1) < r < \frac{\pi}{3}$, and fix $\varepsilon>0$ so that $2r+\varepsilon < \frac{2\pi}{3}$.
By Lemma~\ref{lem:correspondence-vr} we get continuous maps
\[\vr{S^1}{\varepsilon} \xlongrightarrow{\oh} \vr{X}{r+\varepsilon} \xlongrightarrow{\og} \vr{S^1}{2r+\varepsilon}\]
whose composition $\og \circ \oh$ is homotopy equivalent to the inclusion $\iota \colon \vr{S^1}{\varepsilon} \hookrightarrow \vr{S^1}{2r+\varepsilon}$.
Since $2r+\varepsilon<\frac{2\pi}{3}$, it follows that $\vr{S^1}{2r+\varepsilon}\simeq S^1$, that $\vr{S^1}{\varepsilon}\simeq S^1$, and that the inclusion map $\iota$ is a homotopy equivalence~\cite{AA-VRS1}.
Therefore $\og \circ \oh$ is also a homotopy equivalence between spaces homotopy equivalent to circles, and hence nonzero on the fundamental group $H_1$.

If we had $r+\varepsilon < 2\cdot d_\h(X,S^1)$, then there would be an arc of the circle of length longer than $r+\varepsilon$ containing no points from $X$.
The theory of winding fractions (see~\cite{AA-VRS1,AAM})
would then give that the winding fraction $\textrm{wf}(\vr{X}{r+\varepsilon})$ is equal to zero.
By~\cite[Corollary~4.5]{AA-VRS1} or~\cite[Proposition~6.4]{AAM} in the case $X$ is finite, and more generally by~\cite[Proposition~7.5]{AA-VRS1} for $X$ arbitrary, this implies that $\vr{X}{r+\varepsilon}$ is homotopy equivalent to a single point or to a wedge sum of $0$-spheres, and hence that $H_1(\vr{X}{r+\varepsilon})$ is the trivial group.
This would contradict the fact that $\og \circ \oh$ is nonzero on $H_1$.
Therefore, it must be that $r+\varepsilon \ge 2\cdot d_\h(X,S^1)$ for all $2\cdot d_\gh(X,S^1) < r < \frac{\pi}{3}$ and sufficiently small $\varepsilon>0$.
Thus $2\cdot d_\gh(X,S^1) \ge 2\cdot d_\h(X,S^1)$ and $d_\gh(X,S^1) \ge d_\h(X,S^1)$, as desired.

\end{proof}

We also prove Lemma~\ref{lem:circum} here.

\begin{proof}[Proof of Lemma~\ref{lem:circum}]
Let $A\subseteq M$ be compact, with $\diam(A)< \tau$, where $\tau=\rho(M)$ for $\kappa\le 0$ and $\tau=\min\left\{\rho(M),\tfrac{\pi}{2\sqrt{\kappa+1}}\right\}$ for $\kappa>0$.
We must show that the circumcenter $c(A)$ exists in $M$, and that the diameter and circumradius satisfy $\diam(A)\geq2\alpha(n,\kappa)\cdot R(A)$.

Since $\tfrac{\pi}{2\sqrt{\kappa+1}}<\tfrac{2\pi}{3\sqrt{\kappa}}$, the existence of a circumcenter $c(A)$ follows immediately from Theorem \ref{thm:Jung}. For the diameter bound, we now consider the following two cases:

\noindent\textbf{Case I ($\pmb{\kappa\leq 0}$).}
This follows from Theorem~\ref{thm:Jung}, since $\kappa$ is an upper bound on curvature (if the sectional curvatures are upper bounded by a negative number, then they are also upper bounded by zero).

\noindent\textbf{Case II ($\pmb{\kappa>0}$).}
From the definition of
circumradius, we observe that $R(A)\leq\diam(A)$. 

From Theorem~\ref{thm:Jung}, we have 
\begin{align*}
\diam(A)&\geq\tfrac{2}{\sqrt{\kappa}}\sin^{-1}\left(\sqrt{\tfrac{n+1}{2n}}
\sin\left(\sqrt{\kappa}\;R(A)\right)\right) \\
&\geq\tfrac{2}{\sqrt{\kappa}}\sqrt{\tfrac{n+1}{2n}}\ 
\sin\left(\sqrt{\kappa}\;R(A)\right)&&\text{ since }\sin{x}\leq x \\
&=2\sqrt{\tfrac{n+1}{2n}}\ 
\frac{\sin\left(\sqrt{\kappa}\;R(A)\right)}{\sqrt{\kappa}R(A)}R(A).
\end{align*}
Since $\frac{\sin{x}}{x}$ is a decreasing function for $x\in[0,\pi/2]$, the
minimum occurs at the right endpoint of the domain of $x$.
Since $
R(A)\in\left[0,\tfrac{\pi}{2\sqrt{\kappa+1}}\right]$, we get
\[
\frac{\sin\left(\sqrt{\kappa}\;R(A)\right)}{\sqrt{\kappa}R(A)}
\geq\frac{\sin\left(\pi\sqrt{\kappa}/2\sqrt{\kappa+1}\right)}{\pi\sqrt{\kappa}/2\sqrt{\kappa+1}}
=\sin\left(\tfrac{\pi}{2}\sqrt{\tfrac{\kappa}{\kappa+1}}\right)/\left(\tfrac{\pi}{2}\sqrt{\tfrac{\kappa}{\kappa+1}}\right).
\]
Therefore,
\[
\diam(A)\geq2\sqrt{\tfrac{n+1}{2n}}\cdot\left(\sin\left(\tfrac{\pi}{2}\sqrt{\tfrac{\kappa}{\kappa+1}}\right)/\left(\tfrac{\pi}{2}\sqrt{\tfrac{\kappa}{\kappa+1}}\right)\right)\cdot R(A)
=2\alpha(n,\kappa)\cdot R(A).
\]
\end{proof}

\end{document}